\newif\ifagt
\agtfalse

\ifagt
\documentclass{gtpart}     
\agtart
\else
\documentclass{amsart}
\usepackage{fullpage}
\usepackage{hyperref}
\fi

\usepackage{amsrefs}

\usepackage[matrix,arrow,curve,cmtip]{xy}
\newdir{ >}{{}*!/-9pt/\dir{>}}

\title{An abelian embedding for Moore spectra}
\author{N.~P.~Strickland}
\ifagt
\givenname{Neil}
\surname{Strickland}
\fi

\address{
School of Mathematics and Statistics\\
University of Sheffield\\
The Hicks Building\\
Hounsfield Road\\
Sheffield\\
S3 7RH\\
UK}

\email{N.P.Strickland@sheffield.ac.uk}
\urladdr{http://www.shef.ac.uk/nps}

\ifagt

\arxivreference{}
\arxivpassword{}

\volumenumber{}
\issuenumber{}
\publicationyear{}
\papernumber{}
\startpage{}
\endpage{}
\doi{}
\MR{}
\Zbl{}
\received{}
\revised{}
\accepted{}
\published{}
\publishedonline{}
\proposed{}
\seconded{}
\corresponding{}
\editor{}
\version{}
\fi

\DeclareMathOperator{\Ab}{Ab}
\DeclareMathOperator{\ElAb}{ElAb}
\DeclareMathOperator{\Moore}{Moore}
\DeclareMathOperator{\Postnikov}{Postnikov}
\DeclareMathOperator{\Spectra}{Spectra}
\DeclareMathOperator{\ED}{ED}
\DeclareMathOperator{\EED}{EED}
\DeclareMathOperator{\EEED}{EEED}
\DeclareMathOperator{\MD}{MD}
\DeclareMathOperator{\EMD}{EMD}
\DeclareMathOperator{\FreeAb}{FreeAb}
\DeclareMathOperator{\FreeMoore}{FreeMoore}
\DeclareMathOperator{\SPP}{SPP}

\DeclareMathOperator{\Aut}{Aut}
\DeclareMathOperator{\Ext}{Ext}
\DeclareMathOperator{\Hom}{Hom}
\DeclareMathOperator{\Tor}{Tor}
\DeclareMathOperator{\Sq}{Sq}

\DeclareMathOperator{\ev}{ev}

\ifagt
\else
\newcommand{\Z}         {{\mathbb{Z}}}
\newcommand{\Q}         {{\mathbb{Q}}}
\newcommand{\R}         {{\mathbb{R}}}
\newcommand{\C}         {{\mathbb{C}}}
\fi

\renewcommand{\H}       {{\mathbb{H}}}

\newcommand{\al}        {\alpha}
\newcommand{\bt}        {\beta} 
\newcommand{\gm}        {\gamma}
\newcommand{\dl}        {\delta}
\newcommand{\kp}        {\kappa}
\newcommand{\zt}        {\zeta}
\newcommand{\tht}       {\theta}
\newcommand{\sg}        {\sigma}
\newcommand{\lm}        {\lambda}
\newcommand{\om}        {\omega}

\newcommand{\Dl}        {\Delta}
\newcommand{\Sg}        {\Sigma}

\newcommand{\CA}        {{\mathcal{A}}}
\newcommand{\CC}        {{\mathcal{C}}}
\newcommand{\CD}        {{\mathcal{D}}}
\newcommand{\CF}        {{\mathcal{F}}}

\newcommand{\CI}        {{\mathcal{I}}}
\newcommand{\CJ}        {{\mathcal{J}}}

\newcommand{\tH}        {\widetilde{H}}

\newcommand{\Smash}     {\wedge}
\newcommand{\bigWedge}  {\bigvee}
\newcommand{\opp}       {{\text{op}}}
\newcommand{\ot}        {\otimes}
\newcommand{\st}        {\;|\;}
\newcommand{\xra}       {\xrightarrow}
\newcommand{\xla}       {\xleftarrow}
\newcommand{\ov}[1]     {\overline{#1}}
\newcommand{\sse}       {\subseteq}
\newcommand{\Wedge}     {\vee}
\newcommand{\tm}        {\times}

\newcommand{\mar}       {\ar@{ >->}}
\newcommand{\ear}       {\ar@{->>}}

\renewcommand{\:}{\colon}

\newtheorem{theorem}{Theorem}[section]

\newtheorem{lemma}[theorem]{Lemma}
\newtheorem{proposition}[theorem]{Proposition}
\newtheorem{corollary}[theorem]{Corollary}
\theoremstyle{definition}
\newtheorem{remark}[theorem]{Remark}
\newtheorem{definition}[theorem]{Definition}
\newtheorem{example}[theorem]{Example}

\begin{document}

\maketitle 

\begin{abstract}
 We embed the category of Moore spectra as a full subcategory of an
 abelian category, and make some remarks about abelian embeddings of
 various other categories of spectra.
\end{abstract}

\section{Introduction}
 
Given any small additive category $\CI$, we write $[\CI,\Ab]$ for the
category of additive functors $\CI\to\Ab$.  We will call such
categories \emph{diagram categories}.

Let $\Spectra$ denote the homotopy category of spectra (and maps of
degree zero only), and let $\CC$ be a full subcategory of $\Spectra$
(which need not be triangulated).  By an \emph{embedding} of $\CC$ we
mean a full and faithful additive functor from $\CC$ to some diagram
category.

In principle, such embeddings should be common.  A theorem of
Freyd~\cite{fr:sh} guarantees that every triangulated category can be
embedded in an abelian category, and there are other theorems (the
most famous due to Gabriel and Popescu~\cite{gapo:cdc}) saying that
abelian categories can be embedded in diagram categories under fairly
mild conditions.  However, explicit examples are rare.  The best known
cases are as follows:
\begin{itemize}
 \item[(a)] Let $\CC_0$ be the category of rational spectra (those for 
  which each homotopy group $\pi_n(X)$ is uniquely divisible).  Let
  $\CI_0$ be the category with object set $\Z$ and
  \[ \CI_0(n,m) =
      \begin{cases}
       \Q & \text{ if } n=m \\
       0  & \text{ otherwise, }
      \end{cases}
  \]
  so $[\CI_0,\Ab]$ is the category of graded rational vector spaces.
  Then we have an equivalence $\pi_*\:\CC_0\to[\CI_0,\Ab]$.
 \item[(b)] Let $\CC_1$ be the category of Eilenberg-MacLane spectra
  (those with $\pi_k(X)=0$ for $k\neq 0$).  Then $\pi_0\:\CC_1\to\Ab$
  is an equivalence.
 \item[(c)] By a \emph{free Moore spectrum} we mean a $(-1)$-connected
  spectrum $X$ such that $\pi_0(X)$ is free abelian and $H_k(X)=0$ for
  $k>0$.  Let $\CC_2$ be the category of free Moore spectra.  One can
  check that any such spectrum is a wedge of copies of the sphere
  spectrum $S$, and that $\pi_0\:\CC_2\to\FreeAb$ is an equivalence.
 \item[(d)] Let $\CC_3$ be the category of finite spectra.  Let
  $\CI_3$ be the category with object set $\Z$ and
  $\CI_3(n,m)=\pi_{m-n}(S)$, so $[\CI_3,\Ab]$ is the category of
  graded modules over $\pi_*(S)$.  We have a functor
  $\pi_*\:\CC_3\to[\CI_3,\Ab]$, and an old conjecture of Freyd
  suggests that this should be an embedding.  However, there is no
  proof in sight.
 \item[(e)] Say that a spectrum $X$ is \emph{$MU$-injective} if it is
  a retract of $MU\Smash T$ for some $T$, and let $\CC_4$ be the
  category of $MU$-injective spectra.  Let $\CI_4$ be the category
  with object set $\Z$ and $\CI_4(n,m)=[\Sg^{-n}MU,\Sg^{-m}MU]$.
  Define $F\:\CC_4\to[\CI_4,\Ab]$ by
  $F(X)(n)=\pi_0(X\Smash\Sg^{-n}MU)=MU_n(X)$.  It is known that this
  is an embedding.
\end{itemize}

On the other hand, there are various known examples of subcategories
$\CC$ with functors $F\:\CC\to\CA$ (for some diagram category $\CA$)
where the maps $F\:[X,Y]\to\CA(FX,FY)$ are surjective and the kernel
has a straightforward description in terms of well-understood
invariants of $X$ and $Y$.  It is thus natural to ask what is needed
to improve such a functor to an embedding in some more complex diagram
category.  

In this note we will carry out this programme in the simplest possible
case, that of Moore spectra.  (We also have some comments about
slightly larger categories.)  The essentially equivalent case of Moore
spaces of fixed dimension $d>1$ has also been discussed by Baues
in~\cite{ba:ah}*{Section V.3a} (which we read before writing this
paper) and~\cite{ba:ht}*{Chapter 1} (which was brought to our
attention by a referee).  It seems likely that our main theorem could
be obtained by combining his results and spelling out various
definitions more explicitly.  However, we will instead give a direct
approach with more details and context, and greater focus on making
all constructions manifestly natural.  The amount of work involved was
a surprise to the author, at least.  We now describe the main result.

\begin{definition}\label{defn-moore}
 A \emph{Moore spectrum} is a spectrum $X$ such that $\pi_k(X)=0$ for
 $k<0$ and $H_k(X)=0$ for $k>0$.  We write $\Moore$ for the category of
 Moore spectra (a full subcategory of the homotopy category of
 spectra). 
\end{definition}

\begin{remark}\label{rem-notation}
 It is traditional and convenient to use notation $SA$ or $MA$ for a
 Moore spectrum $X$ with $\pi_0(X)\simeq A$, but this suggests that
 $SA$ is a functor of $A$, which is not quite correct.  As we are
 focussing here on questions of naturality we therefore prefer to
 avoid such notation.
\end{remark}

\begin{definition}\label{defn-moore-system}
 A \emph{Moore diagram} is a diagram $M$ of Abelian groups of the form
 \[ \xymatrix{
      A \ar@/^1ex/[r]^\phi & B \ar@/^1ex/[l]^\psi
    }
 \]
 such that $\psi\phi=0$ and $\phi\psi=2.1_B$.  We write $\MD$ for the
 category of Moore diagrams.  We define forgetful functors
 $\al,\bt\:\MD\to\Ab$ by $\al(M)=A$ and $\bt(M)=B$.

 Note that if $M=(A,B,\phi,\psi)$ is a Moore diagram then
 $2\phi=(\phi\psi)\phi=\phi(\psi\phi)=0$ and similarly $2\psi=0$.
 This means that we have induced maps
 \[ A/2 \xra{\phi} B \xra{\psi} A[2] \]
 (where $A[2]$ means $\{a\in A\st 2a=0\}$).  If this sequence is short
 exact, we say that $M$ is an \emph{exact Moore diagram}.  We write
 $\EMD$ for the full subcategory of exact Moore diagrams.
\end{definition}

We will prove that $\Moore$ is equivalent to $\EMD$.  To define the
relevant functor, let $S/2$ denote the cofibre of the degree two
self-map of the sphere spectrum $S$, so we have a cofibration sequence
\[ S \xra{2} S \xra{\rho} S/2 \xra{\bt} S^1 \xra{2} S^1. \]
Let $\eta$ be the Hopf map, which is the unique nontrivial element of
$\pi_1(S)$.  

\begin{theorem}\label{thm-moore-strong}
 There is an equivalence $F\:\Moore\to\EMD$ sending $X\in\Moore$
 to the diagram
 \[ \xymatrix{ 
     [S,X] \ar@/^1ex/[rr]^{(\eta\bt)^*} & &
     [S/2,X] \ar@/^1ex/[ll]^{\rho^*}
    }
 \]
\end{theorem}
\begin{remark}\label{rem-zt-moore} 
 Note here that $[S,X]$ is just another notation for $\pi_0(X)$.  It
 can also be shown that there is a map $\zt\:S^2\to S/2$ (of order $4$)
 such that $\bt\zt=\Sg\eta$, and that the map $\zt^*\:[S/2,X]\to\pi_2(X)$
 is an isomorphism for Moore spectra $X$.  (This will be proved as
 Lemma~\ref{lem-zt-moore}.) Thus, the diagram above is isomorphic to
 the diagram  
 \[ \xymatrix{ 
     \pi_0(X) \ar@/^1ex/[rr]^{(\eta^2)^*} & &
     \pi_2(X) \ar@/^1ex/[ll]^{\rho^*(\zt^*)^{-1}}
    }
 \]
\end{remark}

Theorem~\ref{thm-moore-strong} will be proved in
Section~\ref{sec-moore-spectra}. 

\section{An Ext isomorphism}
\label{sec-Phi}

It will be convenient to have the following result in hand before we
start our main discussion.

\begin{proposition}\label{prop-Phi}
 Let $U$ and $V$ be abelian groups.  Then for any
 extension $E=(V\xra{i}M\xra{p}U)$ there is a well-defined
 homomorphism $\Phi(E)\:U[2]\to V/2$ given by
 $\Phi(E)(u)=i^{-1}(2 p^{-1}(u))+2V$.  This depends only on the
 equivalence class of the extension $E$, and the resulting map
 $\Phi\:\Ext(U,V)/2\to\Hom(U[2],V/2)$ is an isomorphism.  Moreover,
 there is always an exact sequence
 \[ \xymatrix{
  V[2] \mar[r]^i &
  M[2] \ar[r]^p &
  U[2] \ar[r]^{\Phi(E)} &
  V/2 \ar[r]^i &
  M/2 \ear[r]^p &
  U/2.
 } \]
\end{proposition}
\begin{proof}
 First, the homomorphism $\Phi(E)$ arises by applying the Snake Lemma
 to the diagram 
 \[ \xymatrix{
     V \mar[r]^i \ar[d]_2 &
     M \ear[r]^p \ar[d]_2 &
     U \ar[d]_2 \\
     V \mar[r]_i &
     M \ear[r]_p &
     U.
    }
 \]
 This makes it clear that $\Phi$ is well-defined, that it gives a
 six-term exact sequence as claimed, and that it is
 natural in the following sense: for homomorphisms $f\:V\to V'$ and
 $h\:U'\to U$, the map $\Phi(h^*f_*E)$ is the composite
 \[ U'[2] \xra{h} U[2] \xra{\Phi(E)} V/2 \xra{f} V'/2. \]
 It is also clear that $\Phi(E\oplus E')=\Phi(E)\oplus\Phi(E')$.  Now
 suppose we have two parallel extensions $E=(V\to M\to U)$ and
 $E'=(V\to M'\to U)$.  In this context we have a diagonal map
 $\Dl\:U\to U\oplus U$ and a codiagonal map $\nabla\:V\oplus V\to V$,
 so we can form the extension $E''=\nabla_*\Dl^*(E\oplus E')$.  This
 operation is called the Baer sum, and it corresponds to addition in
 the group $\Ext(U,V)$.  In this case the rule for $\Phi(h^*f_*E)$
 just tells us that $\Phi(E'')=\Phi(E)+\Phi(E')$, so $\Phi$ gives a
 homomorphism $\Ext(U,V)\to\Hom(U[2],V/2)$.

 Now consider the case where $2U=0$ and $2V=0$, so $U[2]=U$ and
 $V/2=V$ so we have a map $\Phi\:\Ext(U,V)\to\Hom(U,V)$.  Suppose that
 $\Phi(E)=0$.  The Snake Lemma gives an exact sequence
 \[ V \xra{} M[2] \xra{p} U \xra{\Phi(E)}
     V \xra{i} M/2 \xra{} U,
 \]
 but $\Phi(E)=0$ so the map $p\:M[2]\to U$ is surjective.  This can be
 regarded as an epimorphism of vector spaces over the field $\Z/2$, so
 it has a section, and any such section gives a splitting of $E$, so
 $[E]=0$ in $\Ext(U,V)$.  Thus, the map $\Phi\:\Ext(U,V)\to\Hom(U,V)$
 is injective.

 Now choose a basis
 $\{u_\al\}_{\al\in I}$ for $U$ over $\Z/2$.  Let $\widetilde{U}$ be a
 free module over $\Z/4$ with basis $\{\tilde{u}_\al\}_{\al\in I}$.
 We can then define an extension
 $E_1=(U\xra{i_1}\widetilde{U}\xra{p_1}U)$, where
 $i_1(u_\al)=2\tilde{u}_\al$ and $p_1(\tilde{u}_\al)=u_\al$.  We then
 find that $\Phi(E_1)$ is the identity map $1_U$.  For any map
 $f\:U\to V$ we deduce that $\Phi(f_*E_1)=f_*\Phi(E_1)=f$.  This
 proves that $\Phi$ is surjective (and thus bijective) in the special
 case where $U$ and $V$ both have exponent two.

 We now return to the general case.  The standard six-term exact
 sequences for $\Ext$ tell us that functors of the form $\Ext(U,-)$
 preserve right-exact sequences, and contravariant functors of the
 form $\Ext(-,V)$ convert left-exact sequences to right exact
 sequences.  We can apply this to the sequences $U[2]\xra{}U\xra{2}U$
 and $V\xra{2}V\xra{}V/2$ to get a diagram with right-exact rows and
 columns as follows:
 \[ \xymatrix{
     \Ext(U,V) \ar[r]^2 \ar[d]_2 &
     \Ext(U,V) \ear[r] \ar[d]^2 &
     \Ext(U[2],V) \ar[d]^{2=0} \\
     \Ext(U,V) \ar[r]_2 \ear[d] &
     \Ext(U,V) \ear[r] \ear[d] &
     \Ext(U[2],V) \ear[d]^\simeq \\
     \Ext(U,V/2) \ar[r]_{2=0} & 
     \Ext(U,V/2) \ear[r]_\simeq &
     \Ext(U[2],V/2).
    }
 \]
 Right exactness means that all entries in the last row or column are
 isomorphic to $\Ext(U,V)/2$, as claimed.
\end{proof}

\begin{remark}\label{rem-exponent-two}
 Note that if $2U=0$ or $2V=0$ then (as $\Ext$ is a biadditive
 functor) we have $2.\Ext(U,V)=0$.  It follows that in these cases we
 just have $\Ext(U,V)=\Hom(U[2],V/2)$.
\end{remark}

\begin{proposition}\label{prop-middle}
 Suppose we have the solid part of the following diagram 
 \[ \xymatrix{ 
   V \mar[r]^i \ar[d]_f &
   M \ear[r]^p \ar@{.>}[d]_g & 
   U \ar[d]^h \\
   V' \mar[r]_{i'} &
   M' \ear[r]_{p'} &
   U'
 } \]
 Suppose that the rows are exact, so they define extensions $E$ and
 $E'$.  Consider the following conditions:
 \begin{itemize}
  \item[(a)] There is a map $g$ making the whole diagram commute.
  \item[(b)] $[f_*E]=[h^*E']$ in $\Ext(U,V')$
  \item[(c)] The diagram 
   \[ \xymatrix{
       U[2] \ar[r]^{\Phi(E)} \ar[d]_h &
       V/2 \ar[d]^f \\
       U'[2] \ar[r]_{\Phi(E')} &
       V'/2
      }
   \]
   commutes.
 \end{itemize}
 Then~(a) and~(b) are equivalent, and they imply~(c).  The converse
 also holds provided that $\Ext(U,V')$ has exponent $2$.
\end{proposition}
\begin{proof}
 The equivalence of~(a) and~(b) is standard and straightforward.
 Condition~(c) can be written as $f_*\Phi(E)=h^*\Phi(E')$ (in
 $\Hom(U[2],V'/2)$), so the naturality of $\Phi$ tells us that~(b)
 implies~(c).  If $\Ext(U,V')$ has exponent two then the map
 $\Phi\:\Ext(U,V')\to\Hom(U[2],V'/2)$ is an isomorphism so~(c)
 also implies~(b).
\end{proof}

\section{Some diagram categories}

\begin{definition}\ \\
 \begin{itemize}
  \item[(a)] An \emph{$\eta$-diagram} is just a diagram $A\xra{\eta}C$
   of abelian groups such that $2\eta=0$.  We write $\ED$ for the
   category of $\eta$-diagrams.  There are evident functors
   $\al,\gm\:\ED\to\Ab$ given by $\al(A,C,\eta)=A$ and
   $\gm(A,C,\eta)=C$.  
  \item[(b)] An \emph{extended $\eta$-diagram} (EED) is a diagram of the
   form 
   \[ B \xra{\psi} A \xra{\eta} C \xra{\chi} B \]
   such that $2\eta=0$ and $\psi\chi=0$ and $\chi\eta\psi=2.1_B$.  We
   write $\EED$ for the category of extended $\eta$-diagrams.  There
   are evident projection functors $\al,\bt,\gm\:\EED\to\Ab$ as well
   as a forgetful functor $\pi\:\EED\to\ED$.
 \end{itemize}
\end{definition}

\begin{lemma}\label{lem-EED-rels}
 Let $(A,B,C,\eta,\chi,\psi)$ be an extended $\eta$-diagram.  Then
 $2\psi=0$ and $2\chi=0$ and $4.1_B=0$.  Thus, $\psi$ and $\chi$
 induce maps
 \[ C/2 \xra{\chi} B \xra{\psi} A[2] \]
 whose composite is zero.
\end{lemma}
\begin{proof}
 We have $2\psi=\psi\circ(2.1_B)=\psi\chi\eta\psi$, which is zero
 because $\psi\chi=0$.  Similarly we have
 $2\chi=(2.1_B)\chi=\chi\eta\psi\chi=0$ and
 $4.1_B=(2.1_B)^2=\chi\eta\psi\chi\eta\psi=0$.  It follows that
 $\chi(2C)=(2\chi)(C)=0$, so $\chi$ factors through $C/2$.  Similarly
 $2\psi(B)=0$, so $\psi(B)\leq A[2]$.  
\end{proof}

\begin{definition}\label{defn-EEED}
 An extended $\eta$-diagram is \emph{exact} if the sequence 
 \[ C \xra{2} C \xra{\chi} B \xra{\psi} A \xra{2} A \]
 is exact, or equivalently the associated sequence 
 \[ C/2 \xra{\chi} B \xra{\psi} A[2] \]
 is short exact.  We write $\EEED$ for the full subcategory of exact
 extended $\eta$-diagrams (EEEDs).
\end{definition}

\begin{example}\label{eg-four}
 Let $B$ be any free module over $\Z/4$.  We then have an EEED as
 follows:
 \[ \xymatrix{
  B \ear[r] & B/2 \ar[r]^1 & B/2 \mar[r]^2 & B.
 } \] 
\end{example}

\begin{lemma}\label{lem-EEED-Phi}
 Let $(A,B,C,\eta,\chi,\psi)$ be an EEED, and let
 $\ov{\eta}$ denote the composite 
 \[ \xymatrix{
     A[2] \mar[r] &
     A \ar[r]^{\eta} &
     C \ear[r] &
     C/2.
 } \]
 Then $\Phi$ of the extension $C/2\to B\to A[2]$ is $\ov{\eta}$, so 
 we have natural exact sequences
 \[ \xymatrix{
  C/2 \mar[r]^\chi &
  B[2] \ar[r]^\psi &
  A[2] \ar[r]^{\ov{\eta}} &
  C/2 \ar[r]^\chi &
  B/2 \ear[r]^\psi &
  A[2]
 } \]
\end{lemma}
\begin{proof}
 Consider an element $a\in A[2]$.  Choose an element $b\in B$ with
 $\psi(b)=a$.  To determine $\phi(E)(a)$, we must find
 $\ov{c}\in C/2$ with $\chi(\ov{c})=2b$.  As $2.1_B=\chi\eta\psi$, we
 can take $\ov{c}=\eta\psi(b)=\eta(a)$.  This shows that
 $\phi(E)=\ov{\eta}$ as claimed, and the rest follows from
 Proposition~\ref{prop-Phi}. 
\end{proof}

\begin{definition}\label{defn-xi}
 Let $N=(A,B,C,\eta,\chi,\psi)$ and $N'=(A',B',C',\eta',\chi',\psi')$
 be extended $\eta$-diagrams.  We define
 $\xi\:\Hom(A[2],C'/2)\to\EEED(N,N')$ as follows.  For any map
 $u\:A[2]\to C'/2$ we let $\ov{u}$ denote the composite 
 \[ B  \xra{\psi} A[2] \xra{u} C'/2 \xra{\chi'} B', \]
 and note that $\psi'\ov{u}=0$ and $\ov{u}\chi=0$.  We then let
 $\xi(u)$ be the morphism $N\to N'$ given by the following diagram: 
 \[ \xymatrix{
     B \ar[r]^\psi \ar[d]_{\ov{u}} &
     C \ar[r]^\eta \ar[d]_0 &
     A \ar[r]^\chi \ar[d]^0 &
     B             \ar[d]^{\ov{u}} \\ 
     B' \ar[r]_{\psi'} &
     C' \ar[r]_{\eta'} &
     A' \ar[r]_{\chi'} &
     B'.
    }
 \]
\end{definition}

\begin{proposition}\label{prop-pi}
 The functor $\pi\:\EEED\to\ED$ is essentially surjective and reflects
 isomorphisms.  Moreover, for any $N,N'\in\EEED$ there is a natural
 short exact sequence
 \[ \xymatrix{
     \Hom(\al(N)[2],\gm(N')/2) \mar[r]^(0.6)\xi &
     \EEED(N,N') \ear[r]^(0.45)\pi &
     \ED(\pi(N),\pi(N')).
 } \]
 (In particular, the functor $\pi$ is full.)
\end{proposition}
\begin{proof}
 First consider an $\eta$-diagram $P=(A\xra{\eta}C)$.  Let $\ov{\eta}$
 denote the composite 
 \[ A[2] \xra{} A \xra{\eta} C \xra{} C/2. \]
 By Proposition~\ref{prop-Phi} we can choose an extension
 $E=(C/2\xra{\ov{\chi}}B\xra{\ov{\psi}}A[2])$ with
 $\Phi(E)=\ov{\eta}$.  We define $\chi$ to be the composite
 $C\xra{}C/2\xra{\ov{\chi}}B$, and $\psi$ to be the composite
 $B\xra{\ov{\psi}}A[2]\to A$.  It is then clear that $\psi\chi=0$.
 From the construction of $\Phi$ we see that
 $2.1_B=\ov{\chi}\ov{\eta}\ov{\psi}=\chi\eta\psi$.  We thus have an
 EEED $N=(A,B,C,\eta,\chi,\psi)$ with $\pi(N)=P$, showing that $\pi$
 is essentially surjective.

 Next, suppose we have another EEED
 $N'=(A',B',C',\eta',\chi',\psi')$.  A morphism from $N$ to $N'$ is
 then a commutative diagram as shown on the left below, which
 automatically gives rise to a another diagram as shown on the right: 
 \[ \xymatrix{
      B \ar[r]^\psi \ar[d]_g     &
      A \ar[r]^\eta \ar[d]_f     &
      C \ar[r]^\chi \ar[d]^h     &
      B             \ar[d]^g     & & 
      C/2 \mar[r]^\chi \ar[d]_h  &
      B   \ear[r]^\psi \ar[d]^g  &
      A[2]             \ar[d]^f  \\
      B' \ar[r]_{\psi'}          &
      A' \ar[r]_{\eta'}          &
      C' \ar[r]_{\chi'}          &
      B'                         & & 
      C'/2 \mar[r]_{\chi'}       &
      B' \ear[r]_{\psi'}         &
      A'[2]
    }
 \]
 If $f$ and $h$ are isomorphisms, then so are the induced maps
 $A[2]\to A'[2]$ and $C/2\to C'/2$, so $g$ is also an isomorphism by
 the five-lemma.  This shows that $\pi$ reflects isomorphisms.
 Suppose instead that $f=0$ and $h=0$.  We then have $g\chi=\chi'h=0$
 and $\psi'g=f\psi=0$, so $g$ factors through the cokernel of $\chi$
 and the kernel of $\psi'$.  In other words, there is a unique map
 $\ov{g}\:A[2]\to C'/2$ such that 
 \[ g = (B \xra{\psi} A[2] \xra{\ov{g}} C'/2 \xra{\chi'} B'), \]
 so $(f,g,h)=\xi(\ov{g})$.  This shows that our sequence is left
 exact. 

 Finally, suppose we are given $f$ and $h$ making the middle square on
 the left commute, but not $g$.  Let $E$ and $E'$ denote the
 extensions $C/2\to B\to A[2]$ and $C'/2\to B'\to A'[2]$, so
 $\Phi(E)=\ov{\eta}$ and $\Phi(E')=\ov{\eta}'$ by
 Lemma~\ref{lem-EEED-Phi}.  The relation $h\eta=\eta'f$ implies that
 $h_*\Phi(E)=f^*\Phi(E')$, so Proposition~\ref{prop-middle} tells us
 that there exists $g\:B\to B'$ making the right hand diagram
 commute.  It follows that the remaining squares on the left commute
 also, so we have a morphism $(f,g,h)\in\EEED(N,N')$ with
 $\pi(f,g,h)=(f,h)$ as required.
\end{proof}

The following result proves is evidence that the work done in this
paper is really necessary, and cannot be simplified away.

\begin{proposition}\label{prop-not-split}
 There is no functor $\sg\:\ED\to\EEED$ (additive or otherwise) such
 that the composite $\pi\sg\:\ED\to\ED$ is naturally isomorphic to the
 identity functor.
\end{proposition}

It will be convenient to give the core of the argument as a separate
lemma. 
\begin{lemma}\label{lem-not-split}
 Let $\ElAb$ denote the category of elementary abelian $2$-groups, and
 suppose we have a (not necessarily additive) functor $F\:\ElAb\to\Ab$
 with a natural short exact sequence $U\xra{\chi}F(U)\xra{\psi}U$.
 Then the sequence is naturally split, so $2F(U)=0$.
\end{lemma}
\begin{proof}
 First, the exact sequence shows that $F(0)=0$.  Next, we can choose
 $x\in F(\Z/2)$ with $\psi(x)=1$.  For an arbitrary group $U\in\ElAb$ and 
 $u\in U$ we have $\al_u\:\Z/2\to U$ with $\al_u(1)=u$, and we put
 $\tht(u)=F(\al_u)(x)\in F(U)$.  This defines a natural function (not
 obviously a homomorphism) $\tht\:U\to F(U)$ with $\psi\tht=1$.  Using
 $F(0)=0$ we see that $\tht(0)=0$ for all $U$.  Now consider the element
 $\om_U=\sum_{u\in U}\tht(u)\in F(U)$, which is evidently invariant
 under $\Aut(U)$, so $\psi(\om_U)\in U^{\Aut(U)}$.  Now suppose that
 $|U|>2$, in which case it is easy to see that $U^{\Aut(U)}=0$.  We
 thus have $\psi(\om_U)=0$, so $\om_U=\chi(\om'_U)$ for a unique
 element $\om'_U\in U$.  This is again invariant under $\Aut(U)$ and
 so is zero, which implies that $\om_U=0$.  Now specialise to the case
 $U=(\Z/2)^2=\{0,e_1,e_2,e_1+e_2\}$; the conclusion is that
 $\tht(e_1)+\tht(e_2)+\tht(e_1+e_2)=0$.  Now define
 $\al_{u,v}\:(\Z/2)^2\to U$ by $\al_{u,v}(i,j)=iu+jv$.  By applying
 naturality to this we get $\tht(u)+\tht(v)+\tht(u+v)=0$.
 We can specialise to the case $u=v$ to get $2\tht(u)=0$, and then
 return to the general case to get $\tht(u+v)=\tht(u)+\tht(v)$.  Thus,
 $\tht$ is a natural homomorphism that splits the exact sequence.
\end{proof}
\begin{remark}
 In the statement of the above lemma we have emphasised that $F$ is
 not assumed to be additive, because that is necessary for our
 application.  From the conclusion of the lemma we see that $F$ is
 actually forced to be additive.  More generally, as the referee
 remarked, one can check that any extension of additive functors is
 automatically additive.
\end{remark}

\begin{proof}[Proof of Proposition~\ref{prop-not-split}]
 Suppose we have a functor $\sg\:\ED\to\EEED$, and a natural
 isomorphism $\lm\:1\to\pi\sg$.  Thus, for each object
 $L=(U\xra{\zt}V)\in\ED$ we have an object
 $\sg(L)=(B\xra{\psi}A\xra{\eta}C\xra{\chi}B)\in\EEED$ together with a
 commutative diagram
 \[ \xymatrix{
      U \ar[r]^\zt \ar[d]_{\lm_1}^\simeq & V \ar[d]_\simeq^{\lm_2} \\
      A \ar[r]_\eta & C.
    }
 \]
 We can define a new functor $\sg'\:\ED\to\EEED$ by 
 \[ \sg'(L) = 
     (B\xra{\lm_1^{-1}\psi}U\xra{\zt}V\xra{\chi\lm_2}B)
 \]
 and we find that $\sg'\simeq\sg$ and $\pi\sg'$ is equal (not just
 isomorphic) to the identity.  We will replace $\sg$ by $\sg'$ and
 assume that $\pi\sg=1$.  We now have a functor $F\:\ElAb\to\Ab$ given
 by $F(U)=\bt(\sg(U\xra{1}U))$, and a natural short exact
 sequence $U\xra{\chi}F(U)\xra{\psi}U$ with $2.1_{F(U)}=\chi\psi$.  By
 the Lemma, this must split, so $2.1_{F(U)}=0$, so $\chi\psi=0$.  On
 the other hand, the sequence is short exact, so $\chi$ is injective
 and $\psi$ is surjective.  This clearly gives a contradiction
 whenever $U\neq 0$.  
\end{proof}

We now discuss a slightly different way to think about the category $\EED$.
\begin{definition}\label{defn-CJ}
 We define a small additive category $\CJ$ as follows.  There are
 three objects, denoted by $a$, $b$ and $c$.  There are morphisms 
 \[ b \xla{\rho} a \xla{\eta} c \xla{\bt} b \]
 with $2\rho=0$, $2\eta=0$ and $2\bt=0$.  The full list of morphisms
 is as follows:
 \begin{align*}
  \CJ(a,a) &= \Z & \CJ(a,b) &= (\Z/2)\rho & \CJ(a,c) &= 0 \\
  \CJ(b,a) &= (\Z/2)\eta\bt & \CJ(b,b) &= \Z/4 & \CJ(b,c) &= (\Z/2)\bt \\
  \CJ(c,a) &= (\Z/2)\eta & \CJ(c,b) &= (\Z/2)\rho\eta & \CJ(c,c) &= \Z.
 \end{align*}
 The composition is determined by the rules $\bt\rho=0$ and
 $2.1_b=\rho\eta\bt$. 
\end{definition}

\begin{proposition}\label{prop-CJ-Ab}
 Any EED $N$ gives an additive functor $Q(N)\:\CJ^\opp\to\Ab$ by the rules
 \begin{align*}
  Q(N)(a) &= A & Q(N)(b) &= B & Q(N)(c) &= C \\
  Q(N)(\bt) &= \chi & Q(N)(\eta) &= \eta & Q(N)(\rho) &= \psi,
 \end{align*}
 so 
 \[ Q(N)(b \xla{\rho} a \xla{\eta} c \xla{\bt} b) = 
     (B\xra{\psi} A \xra{\eta} C \xra{\chi} B).
 \]
 Moreover, this construction gives an equivalence from $\EED$ to the
 category $[\CJ^\opp,\Ab]$ of additive contravariant functors from
 $\CJ$ to $\Ab$.
\end{proposition}
\begin{proof}
 Straightforward comparison of definitions.
\end{proof}

From now on we will not distinguish notationally between $N$ and
$Q(N)$. 

\begin{corollary}\label{cor-sixteen}
 For any $x\in\CJ$ we have a representable functor
 $F_x=\CJ(-,x)\:\CJ^\opp\to\Ab$, so $F_x\in\EED$.  For any morphism
 $u\:x\to y$ in $\CJ$ we have a morphism $F_u\:F_x\to F_y$ in $\EED$.
 The object $F_x$ can be displayed as follows:
 \begin{align*} 
     F_b  &= (\xymatrix{
     \Z/4 \ear[r] &
     \Z/2 \ar[r]^1 &
     \Z/2 \mar[r]^2 &
     \Z/4}) \\
     F_c  &= (\xymatrix{
     \Z/2 \ar[r] &
     0 \ar[r] &
     \Z \ear[r] &
     \Z/2}) \\
     F_a &= (\xymatrix{
     \Z/2 \ar[r]_0 &
     \Z \ear[r] &
     \Z/2 \ar[r]_1 &
     \Z/2}).
 \end{align*}
 The morphisms $F_u$ are given by the following square.  The rows are
 $F_b$, $F_c$, $F_a$ and $F_b$ respectively.  The maps from the first
 row to the second comprise the morphism $F_\bt\:F_b\to F_a$.
 Similarly, the second group of vertical maps comprise the morphism
 $F_\eta\:F_c\to F_a$, and the third group comprise
 $F_\rho\:F_a\to F_b$.
 \[ \xymatrix{
     F_b \ar[d]_{F_\bt} &
     \Z/4 \ear[r] \ear[d] &
     \Z/2 \ar[r]^1 \ar[d] &
     \Z/2 \mar[r]^2 \ar[d]^0 &
     \Z/4 \ear[d] \\
     F_c \ar[d]_{F_\eta} &
     \Z/2 \ar[r] \ar[d]_1 &
     0 \ar[r] \ar[d] &
     \Z \ear[r] \ear[d] &
     \Z/2 \ar[d]^1 \\
     F_a \ar[d]_{F_\rho} &
     \Z/2 \mar[d]_2 \ar[r]_0 &
     \Z \ear[r] \ear[d] &
     \Z/2 \ar[r]_1 \ar[d]^1 &
     \Z/2 \mar[d]^2 \\
     F_b &
     \Z/4 \ear[r] &
     \Z/2 \ar[r]_1 &
     \Z/2 \mar[r]_2 &
     \Z/4.
    }
 \]
 By the Yoneda Lemma we have $\EED(F_x,N)=N(x)$.
\end{corollary}
\begin{proof}
 This is a straightforward (if somewhat lengthy) calculation from the
 definitions.  For example, the first row is 
 \[ \xymatrix{
     \CJ(b,b) \ar[r]^{\rho^*} \ar@{=}[d] &
     \CJ(a,b) \ar[r]^{\eta^*} \ar@{=}[d] &
     \CJ(c,b) \ar[r]^{\bt^*}  \ar@{=}[d] &
     \CJ(b,b)                 \ar@{=}[d] \\
     \Z/4       \ear[r] &
     (\Z/2)\rho \ar[r]_\simeq &
     (\Z/2)\rho\eta \mar[r] &
     \Z/4.
    }
 \]
\end{proof}

\begin{definition}\label{defn-degenerations}
 Given subcategories $\CA,\CC\sse\Ab$ we write $\ED[\CA,\CC]$ for the
 subcategory of $\ED$ consisting of $\eta$-diagrams $A\xra{\eta}C$
 with $A\in\CA$ and $C\in\CC$.  We will particularly be interested in
 the categories 
 \begin{align*}
  \CF &= \{A\in\Ab\st 2.1_A \text{ is injective }\} = \{A\st A[2]=0\} \\
  \CD &= \{C\in\Ab\st 2.1_C \text{ is surjective }\} = \{C\st C/2=0\}.
 \end{align*}
\end{definition}
\begin{remark}\label{rem-degenerations}
 Consider an exact extended $\eta$-diagram $M=(A,B,C,\eta,\chi,\psi)$,
 so we have a short exact sequence $C/2\xra{\chi}B\xra{\psi}A[2]$.  If
 $\chi=0$ we must have $C/2=0$ (so $C\in\CD$) and $\psi$ gives an
 isomorphism $B\to A[2]$.  Similarly, if $\psi=0$ then $A\in\CF$ and
 $C/2\simeq B$.  Arguing along these lines, we find that:
 \begin{itemize}
  \item The functor $\pi$ gives an equivalence
   $\{M\st\chi=0\}\to\ED[\Ab,\CD]$.\\  Moreover, $\al$ gives an equivalence
   $\{M\st C=0\}\to\Ab$.
  \item The functor $\pi$ gives an equivalence
   $\{M\st\psi=0\}\to\ED[\CF,\Ab]$.\\  Moreover, $\gm$ gives an equivalence
   $\{M\st A=0\}\to\Ab$.
  \item The functor $\pi$ also gives an equivalence
   $\{M\st B=0\}\to\ED[\CF,\CD]$.
 \end{itemize}
\end{remark}

To complete the picture, we should analyse the case where $\eta=0$.
This is a little more elaborate.
\begin{definition}\label{defn-H}
 We define a category $\SPP$ (short for ``split Postnikov pairs'') as
 follows.  The objects are pairs $(A,C)$ of abelian groups.  The
 morphisms from $(A_0,C_0)$ to $(A_1,C_1)$ are triples $(f,g,u)$ where
 $f\:A_0\to A_1$ and $h\:C_0\to C_1$ and $u\:A_0[2]\to C_1/2$.  The
 identity morphism of $(A,C)$ is $(1_A,1_C,0)$.  The composite of
 \[ (A_0,C_0) \xra{(f_0,h_0,u_0)}
    (A_1,C_1) \xra{(f_1,h_1,u_1)}
    (A_2,C_2) 
 \]
 is $(f_1f_0,h_1h_0,(h_1)_*u_0+(f_0)^*u_1)$.  (We leave it to the
 reader to check that this is associative and unital.)  Next, for
 $(A,C)\in\SPP$ we define $H(A,C)\in\EEED$ to be the diagram
 \[ C/2\oplus A[2] \xra{\psi} A \xra{0} C \xra{\chi} C/2\oplus A[2] \]
 where $\psi(\ov{c},a)=a$ and $\chi(c)=(c+2C,0)$.  Given a morphism
 $(f,h,u)\:(A_0,C_0)\to(A_1,C_1)$ we define
 $g\:C_0/2\oplus A_0[2]\to C_1/2\oplus A_1[2]$ by 
 \[ g(c_0+2C_0,a_0) = (h(c_0)+u(a_0)+2C_1,f(a_0)). \]
 We find that the diagram 
 \[ \xymatrix{
     C_0/2\oplus A_0[2] \ar[rr]^{\psi_0} \ar[d]_g &&
     A_0 \ar[r]^0 \ar[d]_f &
     C_0 \ar[rr]^{\chi_0} \ar[d]^h &&
     C_0/2\oplus A_0[2] \ar[d]^g \\
     C_1/2\oplus A_1[2] \ar[rr]_{\psi_1} &&
     A_1 \ar[r]_0 &
     C_1 \ar[rr]_{\chi_1} &&
     C_1/2\oplus A_1[2]
    }
 \]
 commutes, so we have a morphism
 $(f,g,h)\in\EEED(H(A_0,C_0),H(A_1,C_1))$.  We write $H(f,h,u)$ for
 this morphism.
\end{definition}

\begin{proposition}\label{prop-H-equivalence}
 The above construction gives an equivalence
 $H\:\SPP\to\{M\in\EEED\st\eta=0\}$.
\end{proposition}
\begin{proof}
 We leave it to the reader to check that $H$ gives an isomorphism
 \[ \SPP((A_0,C_0),(A_1,C_1)) \to \EEED(H(A_0,C_0),H(A_1,C_1)), \]
 and that this is compatible with composition.  Thus, $H$ defines a
 full and faithful embedding $\SPP\to\EEED$.  Now consider an EEED $M$
 with $\eta=0$, say $M=(A,B,C,0,\chi,\psi)$.  We then have
 $2.1_B=\psi\eta\chi=0$, so the sequence $C/2\to B\to A[2]$ is a short
 exact sequence of vector spaces over the field $\Z/2$, so it splits.
 Any choice of splitting gives an isomorphism $H(A,C)\to M$.  It
 follows that the essential image of $H$ is $\{M\st\eta=0\}$, as
 claimed.
\end{proof}

We now explain the precise connection between Moore diagrams and
$\eta$-diagrams.  

\begin{proposition}\label{prop-EMD-EEED}
 Let $\EMD'$ denote the full subcategory of $\EEED$ consisting of
 diagrams $(A,B,C,\eta,\chi,\psi)$ for which the map
 $\eta\:A\to C$ is surjective with kernel $2A$.  Then there is an
 equivalence $E\:\EMD\to\EMD'$ given by 
 \[ E(B\xra{\psi}A\xra{\phi}B) = 
     (B\xra{\psi}A\xra{\text{proj}}A/2\xra{\ov{\phi}}B)
 \]
 (where $\ov{\phi}$ is the unique map with
 $\ov{\phi}\circ\text{proj}=\phi$). 
\end{proposition}
\begin{proof}
 The inverse functor is just 
 \[ E^{-1}(B\xra{\psi}A\xra{\eta}C\xra{\chi}B) = 
     (B\xra{\psi}A\xra{\chi\eta}B).
 \]
 We leave all further details to the reader.
\end{proof}

\begin{lemma}\label{lem-al}
 If $N\in\EMD'$ and $N'\in\EEED$ then the map
 $\al\:\ED(\pi(N),\pi(N'))\to\Hom(\al(N),\al(N'))$ is a bijection.
\end{lemma}
\begin{proof}
 If we write $N$ and $N'$ in the usual form, then
 $\ED(\pi(N),\pi(N'))$ is the set of pairs $(f,h)$ making the
 following diagram commute:
 \[ \xymatrix{
      A \ar[d]_f \ear[r]^\eta & C \ar[d]^h \\
      A' \ar[r]_{\eta'} & C'.
    }
 \]
 If we are given an arbitrary homorphism $f\:A\to A'$, we can define
 $h\:C\to C'$ as follows.  Given $a\in A$, we choose $c\in C$ with
 $\eta(c)=a$, then put $h(a)=\eta'(f(c))$.  By assumption we have
 $\ker(\eta)=2A$ and $2\eta'=0$ so $\eta'(f(2A))=0$, which proves that
 $h$ is well-defined.  It is clear that $h$ is the unique homomorphism
 making the diagram commute.  Thus, the projection map
 $(f,h)\mapsto f$ is an isomorphism as claimed.
\end{proof}

\begin{corollary}\label{cor-pi-EMD}
 The functor $\al\:\EMD\to\Ab$ is full and essentially surjective, and
 reflects isomorphisms.  For any $M,M'\in\EMD$ there is a natural short
 exact sequence 
 \[ \Hom(\al(M)[2],\al(M')/2) \xra{}
     \EMD(M,M') \xra{} \Hom(\al(M),\al(M')).
 \]
\end{corollary}
\begin{proof}
 Put $N=E(M)$ and $N'=E(M')$, so these are objects in the full
 subcategory $\EMD'\sse\EEED$.  Proposition~\ref{prop-pi} gives a
 short exact sequence
 \[ \xymatrix{
     \Hom(\al(N)[2],\gm(N')/2) \mar[r]^(0.6)\xi &
     \EEED(N,N') \ear[r]^(0.45)\pi &
     \ED(\pi(N),\pi(N')).
 } \]
 From the definition of $E$ we have $\al(N)=\al(M)$ and
 $\gm(N')/2=\al(M')/2$.  Proposition~\ref{prop-EMD-EEED} identifies
 $\EEED(N,N')$ with $\EMD(M,M')$, and Lemma~\ref{lem-al} identifies
 $\ED(\pi(N),\pi(N'))$ with $\Hom(\al(M),\al(M'))$.  The above short
 exact sequence can thus be rewritten as
 \[ \Hom(\al(M)[2],\al(M')/2) \xra{}
     \EMD(M,M') \xra{} \Hom(\al(M),\al(M')),
 \]
 as claimed.  This shows that $\al:\EMD\to\Ab$ is full.  

 We could prove that $\al$ is essentially surjective by a similar
 process of translation.  More directly, for any abelian group $A$ we
 can take $B=A/2\oplus A[2]$ and $\phi(a)=(a+2A,0)$ and
 $\psi(a+2A,a')=a'$; this gives an exact Moore diagram
 $M=(A,B,\phi,\psi)$ with $\al(M)=A$, proving that $\al$ is
 essentially surjective.

 Similarly, suppose we have a morphism $p\:M\to M'$ in $\EMD$, given
 by maps $f\:A\to A'$ and $g\:B\to B'$, and suppose that $\al(p)=f$ is
 an isomorphism.  We then have a commutative diagram with exact rows
 as follows:
 \[ \xymatrix{
     A/2 \ar[d]_f^\simeq \mar[r]^\phi & 
     B \ar[d]^g \ear[r]^\psi &
     A[2] \ar[d]^f_\simeq \\     
     A'/2 \mar[r]_{\phi'} &
     B' \ear[r]_{\psi'} &
     A'[2]
 } \]
 It follows that $g$ is also an isomorphism, so $p$ is an isomorphism
 in $\EMD$. 
\end{proof}
 
\section{Functors from Spectra}

In order to construct extended $\eta$-diagrams from spectra, we need
some low-dimensional calculations of stable homotopy groups.  In
discussing these, it is helpful to remember that the graded groups
$[X,Y]_*$ are naturally modules over $\pi_*(S)$, and composition is
bilinear.  For example, if we have a map $f\:\Sg^kX\to Y$ then the
following composites are all the same and can be denoted by $\eta.f$:
\[ (\Sg^{k+1}X\xra{\eta\Smash 1_X}\Sg^kX\xra{f}Y) = 
   (\Sg^{k+1}X\xra{\eta\Smash f}Y) = 
   (\Sg^{k+1}X\xra{\Sg f}\Sg Y\xra{\eta\Smash 1_Y}Y).
\]
(We have suppressed any mention of signs here because $2\eta=0$.)
Similar remarks apply with $\eta$ replaced by $2\:S\to S$.  In
particular we have the cofibration sequence
\[ S \xra{2} S \xra{\rho} S/2 \xra{\bt} S^1 \xra{2} S^1 \]
in which adjacent composites vanish, so $2.\rho=0$ and $2.\bt=0$.

\begin{proposition}\label{prop-two}
 The identity map of $S/2$ has order $4$, and the following diagram
 commutes:
 \[ \xymatrix{
      S/2 \ar[d]_2 \ar[r]^\bt & S^1 \ar[d]^\eta \\
      S/2 & S^0 \ar[l]^\rho.
    }
 \]
\end{proposition}
\begin{proof}
 This is classical, probably due to Toda~\cite{to:oic}.  One argument
 is as follows.  We have $S/2=\Sg^{-1}\R P^2$, so
 $H^*(S/2;\Z/2)=\Z/2\{a,b\}$ with $|a|=0$ and $|b|=1$ and
 $\Sq^1(a)=b$.  This gives
 \[ H^*(S/2\Smash S/2) = (\Z/2)\{a\ot a,a\ot b,b\ot a,b\ot b\} \]
 and by the Cartan formula we have $\Sq^2(a\ot a)=b\ot b\neq 0$.
 On the other hand, if we let $\iota$ denote the identity map of
 $S/2$, we see that $S/2\Smash S/2$ is the cofibre of $2\iota$.  If
 $2\iota$ were zero then $S/2\Smash S/2$ would split as
 $S/2\Wedge S^1/2$ and we would have $\Sq^2=0$ in cohomology, which is
 false.  Thus, we must have $2\iota\neq 0$.  On the other hand, the
 cofibration sequence $S\xra{\rho}S/2\xra{\bt}S^1$ gives an exact
 sequence  
 \[ \pi_1(S/2) = [S^1,S/2] \xra{\bt^*} [S/2,S/2]
      \xra{\rho^*} [S,S/2] = \pi_0(S/2)
 \]
 We will take it as given that $\pi_0(S)=\Z$ and $\pi_1(S)=(\Z/2)\eta$
 and $\pi_2(S)=(\Z/2)\eta^2$.  There is a long 
 exact sequence 
 \[ \pi_i(S)\xra{2}\pi_i(S)\xra{\rho_*}\pi_i(S/2)
    \xra{\bt_*}\pi_{i-1}(S) \xra{2} \pi_{i-1}(S).
 \]
 From this we read off that $\pi_0(S/2)=(\Z/2)\rho$ and
 $\pi_1(S/2)=(\Z/2)\rho\eta$.  Thus, our previous exact sequence looks
 like 
 \[ (\Z/2)\rho\eta \xra{\bt^*} [S/2,S/2] \xra{\rho^*} (\Z/2)\rho, \]
 and of course $\rho^*(\iota)=\rho$.  The only way we can have
 $2\iota\neq 0$ is if $\bt^*$ is injective and
 $2\iota=\bt^*(\rho\eta)=\rho\eta\bt$ as claimed.
\end{proof}

\begin{corollary}
 There is a full and faithful additive embedding $T\:\CJ\to\Spectra$
 given by $T(a)=S$ and $T(b)=S/2$ and $T(c)=S^1$. \qed
\end{corollary}

\begin{corollary}
 For any spectrum $X$ there is an exact extended $\eta$-diagram
 $G(X)\:\CJ^\opp\to\Ab$ given by $G(X)(x)=[T(x),X]$ for all $x\in\CJ$.
 In the notation of Corollary~\ref{cor-sixteen}, we have $G(S)=F_a$,
 $G(S/2)=F_b$ and $G(S^1)=F_c$.
\end{corollary}
\begin{proof}
 It is formal that the given rule defines an additive functor
 $\CJ^\opp\to\Ab$, or in other words an extended $\eta$ diagram.
 The sequence
 \[ T(a) \xra{2} T(a) \xra{T(\rho)} T(b) 
     \xra{T(\bt)} T(c) \xra{2} T(c)
 \]
 is the cofibration sequence
 \[ S \xra{2} S \xra{\rho} S/2 \xra{\bt} S^1 \xra{2} S^1 \]
 so when we apply $[-,X]$ we get an exact sequence
 \[ G(X)(a) \xla{2} G(X)(a) \xla{\psi} G(X)(b) \xla{\chi}
     G(X)(c) \xla{2} G(X)(c) 
 \]
 or equivalently 
 \[ \xymatrix{ 
     G(X)(c)/2 \mar[r]^\chi &
     G(X)(b) \ear[r]^\psi &
     G(X)(a)[2].
 } \]
 Thus, we actually have an exact extended $\eta$-diagram.  As $T$ is
 full and faithful we have
 \[ G(T(w))(x) = [T(x),T(w)] = \CJ(x,w) = F_w(x), \]
 so $G(T(w))=F_w$.  This gives $G(S)=F_a$, $G(S/2)=F_b$ and
 $G(S^1)=F_c$ as claimed.
\end{proof}

\begin{remark}
 We will prove as Corollary~\ref{cor-G-surjective} that $G$ is
 essentially surjective.  Lemma~\ref{lem-H-one} describes a
 subcategory on which it reflects isomorphism.
 Corollary~\ref{cor-CW-maps} gives a smaller subcategory where $G$ is
 full, together with a good description of the kernel of the map
 $G\:[X,Y]\to\EEED(G(X),G(Y))$ in that context.  
\end{remark}

\begin{example}
 Given any pair of abelian groups $A$ and $C$, we have 
 Eilenberg-MacLane spectra $HA$ and $HC$ and we find that
 $G(HA\Wedge\Sg HC)=G(HA)\oplus G(\Sg HC)$ is the diagram $H(A,C)$
 defined in Definition~\ref{defn-H}.  However,
 the resulting map 
 \[ [HA_0\Wedge\Sg HC_0,HA_1\Wedge\Sg HC_1] \to 
     \SPP((A_0,C_0),(A_1,C_1)) = 
     \EEED(H(A_0,C_0),H(A_1,C_1))
 \]
 is not an isomorphism in general.  To cure this, let $\SPP^+$ be the
 category with the same objects as $\SPP$, and morphisms
 \[ \SPP^+((A_0,C_0),(A_1,C_1)) = 
     \Hom(A_0,A_1) \tm \Hom(C_0,C_1) \tm \Ext(A_0,C_1),
 \]
 with composition
 \[ (f_1,g_1,u_1)\circ (f_0,g_0,u_0) = 
     (f_1f_0,g_1g_0,(g_1)_*(u_0)+f_0^*(u_1)).
 \]
 Using the Hurewicz and Universal Coefficient theorems, we obtain
 natural isomorphisms
 \begin{align*}
  [HA_0,HA_1] &\simeq \Hom(A_0,A_1) &
  [\Sg HC_0,HA_1] &\simeq 0 \\
  [HA_0,\Sg HC_1] &\simeq \Ext(A_0,C_1) &
  [\Sg HC_0,\Sg HC_1] &\simeq \Hom(C_0,C_1).
 \end{align*}
 Using this we see that there is a full and faithful embedding
 $\tH\:\SPP^+\to\Spectra$ given by $\tH(A,C)=HA\Wedge\Sg HC$.  The map
 $\Phi\:\Ext(A_0,C_1)\to\Hom(A_0[2],C_1/2)$ from
 Proposition~\ref{prop-Phi} gives rise to a functor $\SPP^+\to\SPP$,
 which we also denote by $\Phi$.  We then find that the following
 diagram commutes up to natural isomorphism:
 \[ \xymatrix{
     \SPP^+ \ar[r]^{\tH} \ar[d]_\Phi & \Spectra \ar[d]^G \\
     \SPP \ar[r]_H & \EEED.
    }
 \]
\end{example}

\begin{definition}\label{defn-K}
 For any spectra $X$ and $Y$, we define
 \begin{align*}
  L(X,Y) &= \ker(G\:[X,Y]\to\EEED(G(X),G(Y))) \\
  K(X,Y) &= \ker(\pi G\:[X,Y] \to \ED(\pi G(X),\pi G(Y))) \\
         &= \{f\:X\to Y\st \pi_0(f)=0 \text{ and } \pi_1(f)=0\}.
 \end{align*}
 For any $f\in K(X,Y)$ we form a cofibration sequence $X\to Y\to Cf\to\Sg X$, 
 giving an exact sequence
 \[ \pi_1(X) \xra{f_*=0} \pi_1(Y) \xra{} \pi_1(Cf) \xra{}
    \pi_0(X) \xra{f_*=0} \pi_0(Y)
 \]
 and thus an extension 
 \[ E(f) = (\xymatrix{ \pi_1(Y) \mar[r] & \pi_1(Cf) \ear[r] & \pi_0(X) }).
 \]
 Standard results about uniqueness of cofibration sequences show that this is
 well-defined up to equivalence, we have a well-defined class
 $\om(f)=[E(f)]\in\Ext(\pi_0(X),\pi_1(Y))$.   
\end{definition}

\begin{lemma}\label{lem-om-natural}
 The map $\om\:K(X,Y)\to\Ext(\pi_0(X),\pi_1(Y))$ is a homomorphism.
 Moreover, given maps $W\xra{p}X\xra{f}Y\xra{q}Z$ with $f\in K(X,Y)$
 we have 
 \[ \om(qfp) = \pi_1(q)_*\pi_0(p)^*\om(f) \in\Ext(\pi_0(W),\pi_1(Z)),
 \]
\end{lemma}
\begin{proof}
 First suppose we have $X\xra{f}Y\xra{q}Z$ with $f\in K(X,Y)$.  The
 octahedral axiom then gives a commutative diagram
 \[ \xymatrix{
      Y \ar[r] \ar[d]_q & Cf \ar[r] \ar[d] & \Sg X \ar@{=}[d] \\
      Z \ar[r] & C(qf) \ar[r] & \Sg X.
    }
 \]
 We can apply $\pi_1$ to get a morphism of extensions, and then use
 the first part of Proposition~\ref{prop-middle} to deduce that
 $\om(qf)=\pi_1(q)_*\om(f)$.  A similar argument shows that
 $\om(fp)=\pi_0(p)^*\om(f)$.  Now consider the diagonal map
 $X\to X\Wedge X$ and the codiagonal map $\nabla\:Y\Wedge Y\to Y$.
 Given another map $f'\in K(X,Y)$ we have 
 \[ \om(f+f') = 
    \om(\nabla\circ(f\Wedge f')\circ\Dl) = 
    \nabla_*\Dl^*(\om(f)\oplus\om(f')) = \om(f)+\om(f'),
 \] 
 so $\om$ is a homomorphism.   
\end{proof}

\begin{proposition}\label{prop-om-xi}
 There is a natural commutative diagram as follows, in which the
 rectangle and the squares are pullbacks:
 \[ \xymatrix{
     L(X,Y) \mar[d] \ar[r]^(0.4){\om'} &
     2.\Ext(\pi_0(X),\pi_1(Y)) \ar[r] \mar[d] &
     0 \ar[d] \\
     K(X,Y) \ar[r]_(0.4)\om \mar[d] &
     \Ext(\pi_0(X),\pi_1(Y)) \ear[r]_(0.4)\Phi &
     \Hom(\pi_0(X)[2],\pi_1(Y)/2) \ar[r] \mar[d]_\xi &
     0 \ar[d] \\
     [X,Y] \ar[rr]_G & &
     \EEED(G(X),G(Y)) \ar[r]_\pi &
     \ED(\pi G(X),\pi G(Y)).
    }
 \]
 Moreover:
 \begin{itemize}
  \item[(a)] If the maps $\pi G$ and $\om$ are surjective, then so is
   $G$.
  \item[(b)] If $\om$ is an isomorphism then so is $\om'$.
  \item[(c)] If $\om$ is an isomorphism and
   $2.\Ext(\pi_0(X),\pi_1(Y))=0$, then $G$ is injective.
 \end{itemize}
\end{proposition}
\begin{proof}
 The bottom right square is a pullback (with $\pi$ surjective) by
 Proposition~\ref{prop-pi}, and the top middle square is a pullback
 (with $\Phi$ surjective) by Proposition~\ref{prop-Phi}.  We next show
 that the bottom left rectangle commutes.  Consider an element
 $f\in K(X,Y)$, so $\pi G(f)=0$.  As $\xi$ is the kernel of $\pi$, we
 have $G(f)=\xi(s)$ for some $s\:\pi_0(X)[2]\to \pi_1(Y)/2$.  We need
 to show that $s=\Phi(\om(f))$.  Consider an element
 $w\in\pi_0(X)[2]$.  By inspecting the definitions, we see that $s(w)$
 can be described as follows: we choose any $v\:S/2\to X$ with
 $v\rho=w$, then choose any $u\in\pi_1(Y)$ with $u\bt=fv\:S/2\to Y$,
 then $s(w)$ is the coset $u+2\pi_1(Y)$.  The best way to choose $u$
 is as follows.  Let $h\:\Sg^{-1}Cf\to X$ be the connecting map.  As
 $\pi_0(f)=0$ we have $fw=0$, so we can choose $z\:S\to\Sg^{-1}Cf$
 with $hz=w=v\rho$.  This gives the solid part of the diagram below,
 in which the rows are cofibration sequences:
 \[ \xymatrix{
   S \ar[r]^\rho \ar[d]_z &
   S/2 \ar[r]^\bt \ar[d]_v & 
   S^1 \ar[r]^2 \ar@{.>}[d]_u &
   S^1 \ar[d]^{\Sg z} \ar[r]^{\Sg\rho} &
   S^1/2 \ar[d]^{\Sg v} \\
   \Sg^{-1}Cf \ar[r]_h &
   X \ar[r]_f & 
   Y \ar[r]_g &
   Cf \ar[r]_{\Sg h} &
   \Sg X.
 } \]
 By the axioms for a triangulated category, there is a map $u$ making
 everything commute.  In particular we have $u\bt=fv$ so $u$
 represents $s(w)$.  On the other hand, we can also apply $\pi_1$ to
 obtain a commutative diagram
 \[ \xymatrix{
     \Z \mar[r]^2 \ar[d]_u & \Z \ear[r] \ar[d]_z & \Z/2 \ar[d]^w \\
     \pi_1(Y) \mar[r] & \pi_1(Cf) \ear[r] & \pi_0(X)
 } \]
 from which we can read off the fact that $\Phi(\om(f))(w)$ is
 represented by $u$.  Thus $\xi\Phi\om=G|_{K(X,Y)}$ as claimed.
 Moreover, we have $f\in L(X,Y)$ iff $G(f)=0$ iff $\xi\Phi\om(f)=0$
 iff $\Phi\om(f)=0$ iff
 $\om(f)\in\ker(\Phi)=2.\Ext(\pi_0(X),\pi_1(Y))$.  We can thus define
 $\om'$ to be the restriction of $\om$, and this makes the top left
 square a pullback.  It also follows formally from the definition
 of $K(X,Y)$ that the bottom left rectangle is a pullback.  The
 auxilary statements~(a), (b) and~(c) now follow easily by chasing the
 diagram.  
\end{proof}

\begin{lemma}\label{lem-Q}
 There is a finite spectrum $Q$ fitting in a diagram as follows, in
 which the rows and columns are cofibration sequences.
 \[ \xymatrix{
     & S \ar@{=}[r] \ar[d] & 
     S \ar[d]^\tht \\
     S^2 \ar@{=}[d] \ar[r]^\lm &
     S/\eta \ar[d]_\dl \ar[r] &
     Q \ar[r] \ar[d] &
     S^3 \ar@{=}[d] \\
     S^2 \ar[r]_2 &
     S^2 \ar[d]_\eta \ar[r]_\rho &
     S^2/2 \ar[d]^\xi \ar[r]_\bt &
     S^3 \\ &
     S^1 \ar@{=}[r] &
     S^1.
    }
 \]
 Moreover, the unit map $S\to H$ factors uniquely through
 $\tht\:S\to Q$, and the fibre of the resulting map $Q\to H$ is
 $2$-connected. 
\end{lemma}

This is a fairly well-known example in low-dimensional stable homotopy
theory, and is often characterised by the ``cell diagram''
\[ \xymatrix @!R @!C {
 \bullet \ar@/^1ex/[rr]^{\text{Sq}^2} &
 \cdot & 
 \bullet \ar@/^1ex/[r]^{\text{Sq}^1} &
 \bullet
}\]
However, it will be simpler to prove the precise properties that we
need rather than trying to extract them from the literature.

\begin{proof}
 The third row is a familiar cofibration sequence.  We define $S/\eta$
 to be the cofibre of the map $\eta\:S^1\to S$.  (This can also be
 described as $\Sg^{-2}\C P^2$.)  We let $\dl$ denote the connecting
 map, so that the second column in the diagram is also a cofibration
 sequence.  Using this cofibration sequence and our knowledge of
 $\pi_{\leq 2}(S)$, we see that there is a unique element
 $\lm\in\pi_2(S/2)$ with $\dl\lm=2$, so that the left square commutes.
 We also find that $\pi_0(S/\eta)=\Z$ and $\pi_1(S/\eta)=0$ and
 $\pi_2(S/\eta)=\Z\lm$.

 Next, we define $Q$ to be the cofibre of $\lm$, so the second row is
 another cofibration sequence.  The octahedral axiom now tells us that
 the third column can be filled in as claimed.  If we apply the
 functor $[-,H]$ to that column we see that the unit map $S\to H$
 factors uniquely through $\tht$.  We can also apply $\pi_*$ to the
 second row to see that $\pi_0(Q)=\Z$ and $\pi_1(Q)=\pi_2(Q)=0$.  Now
 let $F$ be the fibre of the map $Q\to H$.  We then have $\pi_0(F)=0$
 and $\pi_k(F)=\pi_k(Q)$ for $k>0$, so in particular
 $\pi_1(F)=\pi_2(F)=0$ as claimed.
\end{proof}
\begin{remark}\label{rem-Q-geometric}
 It is possible to give a more geometric construction, as follows.
 Let $\H$ denote the space of quaternions, and $\H_0$ the subspace of
 purely imaginary ones.  There are evident maps
 \[ P_\R(\H\oplus\H_0) \to P_\R(\H\oplus\H) \to P_\H(\H\oplus\H) 
     \simeq\H_{\infty}\simeq S^4.
 \]
 The composite sends $P_\R(0\oplus\H_0)$ to the basepoint.  Thus, we
 have an induced map $f$ from the quotient space
 $P=P_\R(\H\oplus\H_0)/P_\R(0\oplus\H_0)\simeq\R P^6/\R P^2$ to $S^4$.
 We put $Q=\Sg^{-4}Cf$.  

 We have $H^*(P;\Z/2)=(\Z/2)\{a_3,a_4,a_5,a_6\}$ with $\Sq^1(a_3)=a_4$
 and $\Sq^1(a_5)=a_6$ and $\Sq^2(a_3)=a_5$.  One can check that $f$
 restricts to give a map $P_\R(\C\oplus\H_0)/P_\R(0\oplus H_0)\to S^4$
 that is generically bijective, and therefore sends the generator of
 $H^4(S^4;\Z/2)$ to $a_4$.  It follows that $H^*(Q;\Z/2)$ has the form
 $(\Z/2)\{b_0,b_2,b_3\}$ with $|b_i|=i$ and $\Sq^2(b_0)=b_2$ and
 $\Sq^1(b_2)=b_3$.  From this one can deduce that $Q$ fits into a
 system of cofibration sequences as indicated.  We leave further
 details of this approach to the reader.
\end{remark}

We next recall a standard fact for ease of reference.
\begin{lemma}\label{lem-connective-hurewicz}
 Let $X$ be a connective spectrum.
 \begin{itemize}
  \item[(a)] If $H_*(X)=0$ then $X$ is contractible.
  \item[(b)] If $i$ is the smallest integer such that $H_i(X)\neq 0$,
   then the Hurewicz map $\pi_i(X)\to H_i(X)$ is an isomorphism.
 \end{itemize}
\end{lemma}
\begin{proof}
 First suppose that $H_*(X)=0$.  By assumption $X$ is connective, so
 there exists $n$ such that $\pi_j(X)=0$ for all $j<n$.  The Hurewicz
 theorem then tells us that $\pi_n(X)\simeq H_n(X)=0$.  This allows us
 to apply the Hurewicz theorem one dimension higher, so
 $\pi_{n+1}(X)\simeq H_{n+1}(X)=0$.  We can continue this inductively
 to see that $\pi_i(X)=0$ for all $X$, so $X$ is contractible.
 The proof for~(b) is essentially the same.
\end{proof}
\begin{remark}
 Connectivity is certainly required here; for example, the spectrum
 $X=KU\Smash S/2$ is not contractible but has $H_*(X)=0$.
\end{remark}

\begin{lemma}\label{lem-H-one}
 Let $\CC$ be the category of $(-1)$-connected spectra $X$ with
 $H_k(X)=0$ for $k>1$.  Then for $X\in\CC$ there is a natural short
 exact sequence
 \[ \xymatrix{ 
     \pi_0(X)/2 \mar[r]^\eta &
     \pi_1(X) \ear[r]^h &
     H_1(X).
 } \]
 Moreover, the functor $G\:\CC\to\EEED$ reflects isomorphism.
\end{lemma}
\begin{proof}
 Let $Q$ be as in Lemma~\ref{lem-Q}, and let $F$ be the fibre of the
 map $Q\to H$.  As $X$ is $(-1)$-connected and $F$ is $2$-connected we
 see that $F\Smash X$ is $2$-connected and thus that the map
 $\pi_k(Q\Smash X)\to\pi_k(H\Smash X)=H_k(X)$ is an isomorphism for
 $k\leq 2$.

 Next, from the third column in Lemma~\ref{lem-Q} we have an exact
 sequence
 \[ \pi_2(Q\Smash X) \to \pi_1((S^1/2)\Smash X) \to 
     \pi_1(X) \to \pi_1(Q\Smash X) \to \pi_0((S^1/2)\Smash X)
 \]
 The first group here is $H_2(X)$, which is zero by assumption.  The
 second is $\pi_0((S/2)\Smash X)$, which is easily shown to be
 $\pi_0(X)/2$.  The fourth is $H_1(X)$, and the last is zero for
 connectivity reasons.  We thus have an exact sequence
 \[ 0 \to \pi_0(X)/2 \xra{\eta} \pi_1(X) \xra{} H_1(X) \to 0 \]
 as claimed.

 Now suppose we have a map $f\:X\to Y$ in $\CC$ such that $G(f)$ is an
 isomorphism.  This means that $\pi_0(f)=\al(G(f))$ and
 $\pi_1(f)=\gm(G(f))$ are isomorphisms, so $H_1(f)$ is also an
 isomorphism by our short exact sequence.  We also know that
 $H_0(f)=\pi_0(f)$ by the Hurewicz theorem, and $H_k$ vanishes for
 $k<0$ or $k>1$, so $H_*(f)$ is an isomorphism in all degrees.  As $X$
 and $Y$ are connective this means that $f$ is an equivalence.  Thus,
 the functor $G\:\CC\to\EEED$ reflects isomorphism.
\end{proof}

\section{Moore spectra}
\label{sec-moore-spectra}

\begin{lemma}\label{lem-moore-homology}
 If $X$ is a Moore spectrum then the Hurewicz map $\pi_0(X)\to H_0(X)$
 is an isomorphism, and $H_k(X)=0$ for all $k\neq 0$.
\end{lemma}
\begin{proof}
 We have $\pi_k(X)=0$ for $k<0$ by assumption, so the Hurewicz theorem
 tells us that $H_k(X)=0$ for $k<0$ and that $\pi_0(X)\simeq H_0(X)$.
 We also have $H_k(X)=0$ for $k>0$ by assumption.
\end{proof}
\begin{corollary}\label{cor-reflects}
 The functor $\pi_0\:\Moore\to\Ab$ reflects isomorphism.
\end{corollary}
\begin{proof}
 Suppose that $f\:X\to Y$ is a morphism of Moore spectra such that
 $\pi_0(f)$ is an isomorphism, and let $Cf$ be the cofibre.  Using the
 lemma we see that $H_*(f)$ is an isomorphism, so $H_*(Cf)=0$.  As $X$
 and $Y$ are connective, the same is true of $Cf$.
 Lemma~\ref{lem-connective-hurewicz} therefore tells us that $Cf$ is
 contractible, so $f$ is an equivalence.
\end{proof}

\begin{definition}\label{defn-free-moore}
 A \emph{free Moore spectrum} is a Moore spectrum $X$ for which the
 group $\pi_0(X)$ is free.  We write $\FreeMoore$ for the category of
 such spectra.
\end{definition}

\begin{lemma}\label{lem-free-moore}
 \begin{itemize}
  \item[(a)] The functor $\pi_0\:\FreeMoore\to\FreeAb$ is an
   equivalence.
  \item[(b)] Every free Moore spectrum is a wedge of some family of
   copies of $S$.
  \item[(c)] If $X$ is a free Moore spectrum, then the map
   $\pi_0\:[X,Y]\to\Hom(\pi_0(X),\pi_0(Y))$ is an isomorphism for all
   $Y$, as is the map $\eta\:\pi_0(X)/2\to\pi_1(X)$.
 \end{itemize}
\end{lemma}
\begin{proof}
 Let $X$ be a free Moore spectrum.  Choose a family of maps
 $\{g_i\:S\to X\}_{i\in I}$ giving a basis for the free abelian group
 $\pi_0(X)$.  Put $Q=\bigWedge_IS$ and let $g\:Q\to X$ be given by
 $g_i$ on the $i$'th summand.  Let $Y$ be the cofibre of $g$, so $Y$
 is connective.  By construction $g$ gives an isomorphism from
 $\pi_0(Q)=H_0(Q)$ to $\pi_0(X)=H_0(X)$, and both $Q$ and $X$ have
 homology concentrated in degree zero, so $H_*(Y)=0$.  We thus have
 $Y=0$, so $g$ is an equivalence.  This proves~(b).  Next, the
 category of those spectra for which~(c) holds clearly contains $S$
 and is closed under coproducts, so it contains $Q$ and therefore
 $X$.  By specialising to the case where $Y$ is also a free Moore
 spectrum we obtain~(a).
\end{proof}

\begin{definition}\label{defn-presentation}
 A \emph{presentation} of a Moore spectrum $X$ is a cofibration
 sequence $P\xra{f}Q\xra{g}X\xra{h}\Sg P$, where both $P$ and $Q$ are
 free Moore spectra.
\end{definition}
\begin{remark}
 As $H_1(X)=0$ by the definition of a Moore spectrum, we find that
 the map $\pi_0(f)=H_0(f)$ must be injective.  
\end{remark}

\begin{lemma}\label{lem-presentation}
 Every Moore spectrum admits a presentation.
\end{lemma}
\begin{proof}
 Let $X$ be a Moore spectrum.  Choose a free abelian group $Q_0$ and a
 surjective map $g_0\:Q_0\to\pi_0(X)$.  By Lemma~\ref{lem-free-moore}
 we can find a free Moore spectrum $Q$ with $\pi_0(Q)=Q_0$, and a map
 $g\:Q\to X$ with $\pi_0(g)=g_0$.  We can then form a cofibration
 sequence $P\xra{f}Q\xra{g}X\xra{h}\Sg P$ for some spectrum $P$.  This
 gives long exact sequences of homotopy and homology groups, showing
 that $P$ is a Moore spectrum with $\pi_0(P)=\ker(g_0)$.  This is a
 subgroup of the free abelian group $Q_0$, so it is again free, so $P$
 is a free Moore spectrum as required.
\end{proof}

\begin{lemma}\label{lem-F}
 If $X$ is a Moore spectrum then the map $\eta\:\pi_0(X)/2\to\pi_1(X)$
 is an isomorphism.  Thus, the diagram $G(X)$ lies in the subcategory
 $\EMD'\sse\EEED$.  The corresponding exact Moore diagram
 $F(X)=E^{-1}G(X)$ is 
 \[ \xymatrix{ 
     [S,X] \ar@/^1ex/[rr]^{(\eta\bt)^*} & &
     [S/2,X] \ar@/^1ex/[ll]^{\rho^*}
    }
 \]
 as in Theorem~\ref{thm-moore-strong}.
\end{lemma}
\begin{proof}
 Choose a presentation $P\xra{f}Q\xra{g}X\xra{h}\Sg P$.  This gives an
 exact sequence
 \[ \pi_1(P) \xra{\pi_1(f)} \pi_1(Q) \xra{} \pi_1(X) \xra{}
    \pi_0(P) \xra{\pi_0(f)} \pi_0(Q) \xra{} \pi_0(X) \xra{} 0.
 \]
 Here $\pi_0(f)$ is injective by the definition of a presentation, so
 $\pi_1(X)$ is the cokernel of $\pi_1(f)$, and also $\pi_0(X)$ is the
 cokernel of $\pi_0(f)$.  Tensoring with $\Z/2$ is right exact, so
 $\pi_0(X)/2$ is the cokernel of the map $\pi_0(P)/2\to\pi_0(Q)/2$
 induced by $\pi_0(f)$.  We thus have a commutative diagram with right
 exact rows:
 \[ \xymatrix{
     \pi_0(P)/2 \ar[r] \ar[d]_{\eta} & 
     \pi_0(Q)/2 \ar[r] \ar[d]_{\eta} & 
     \pi_0(X)/2        \ar[d]_{\eta} \\
     \pi_1(P) \ar[r] &
     \pi_1(Q) \ar[r] &
     \pi_1(X).
    }
 \]
 The first two vertical maps are isomorphisms by
 Lemma~\ref{lem-free-moore}(c), so the third one is also an
 isomorphism, as claimed.  The rest is clear from this.
\end{proof}

\begin{lemma}\label{lem-F-reflects}
 The functor $F\:\Moore\to\EMD$ is essentially surjective and
 reflects isomorphisms.
\end{lemma}
\begin{proof}
 First, let $f\:X\to Y$ be a map of Moore spectra such that $F(f)$ is
 an isomorphism.  Then $\pi_0(f)=\al(F(f))$ is an isomorphism, so $f$
 is an isomorphism by Corollary~\ref{cor-reflects}.  This means that
 $F$ reflects isomorphism, as claimed.  (This could also be deduced
 from Lemma~\ref{lem-H-one}.)

 Now let $M=(A,B,\phi,\psi)$ be any exact Moore diagram.  Choose a
 surjective homomorphism $g_0\:Q_0\to A$ with $Q_0$ a free abelian
 group, and let $f_0\:P_0\to Q_0$ be the kernel of $g_0$.  Note that
 $P_0$ is a subgroup of the free abelian group $Q_0$ and so is free.
 We can thus find a map $f\:P\to Q$ of free Moore spectra with
 $\pi_0(P)=P_0$, $\pi_0(Q)=Q_0$ and $\pi_0(f)=f_0$.  Let $X$ be the
 cofibre of $f$; we find that $X$ is a Moore spectrum with
 $\pi_0(X)=A$, or in other words $\al(F(X))=\al(M)$.
 Corollary~\ref{cor-pi-EMD} tells us that the functor
 $\al\:\EMD\to\Ab$ is full, so we can choose $p\:F(X)\to M$ with
 $\al(p)=1_A$.  As $\al$ reflects isomorphism, we deduce that $p$ is
 an isomorphism.  This proves that $F$ is essentially surjective.
\end{proof}

\begin{proposition}\label{prop-moore-maps}
 Let $X$ be a Moore spectrum, and let $X'$ be an arbitrary spectrum.
 Then the map 
 \[ \pi G \: [X,X'] \to \ED(\pi G(X),\pi G(X')) \]
 is surjective.  The kernel of this map is $K(X,X')$, and the map
 $\om \: K(X,X')\to\Ext(\pi_0(X),\pi_1(X'))$ is a bijection.
\end{proposition}
\begin{proof}
 Note that $G(X)\in\EMD'$ by Lemma~\ref{lem-F}, so the map
 \[ \al \: \ED(\pi G(X),\pi G(X')) \to \Hom(\pi_0(X),\pi_0(X')) \]
 is a bijection by Lemma~\ref{lem-al}.  It will thus suffice to
 show that the map $\pi_0\:[X,X']\to\Hom(\pi_0(X),\pi_0(X'))$ is
 surjective, and to identify its kernel.

 We will use the standard notation $N=(A,B,C,\eta,\chi,\psi)$ for
 $G(X)$ (so $A=\pi_0(X)$ and so on), and similarly for $G(X')$.
 Choose a presentation $P\xra{f}Q\xra{g}X\xra{h}\Sg P$, and write
 $P_0=\pi_0(P)$ and so on.  We then have a diagram 
 \[ \xymatrix{
     [\Sg Q,X'] \ar[r]^{f^*} \ar[d]_{\pi_1}^\simeq &
     [\Sg P,X'] \ar[r]^{h^*} \ar[d]_{\pi_1}^\simeq &
     [X,X'] \ar[r]^{g^*} \ear[d]^{\pi_0} &
     [Q,X'] \ar[r]^{f^*} \ar[d]^{\pi_0}_\simeq &
     [P,X']              \ar[d]^{\pi_0}_\simeq \\
     \Hom(Q_0,C') \ar[r]_{f_0^*} &
     \Hom(P_0,C') &
     \Hom(A,A') \mar[r]_{g_0^*} &  
     \Hom(Q_0,A') \ar[r]_{f_0^*} &
     \Hom(P_0,A').
    }
 \]
 The arrows marked as isomorphisms are indeed isomorphisms, by
 Lemma~\ref{lem-free-moore}(c).  The top row is exact because it comes
 from a cofibration sequence.  The right hand half of the bottom row is left
 exact, because it arises from the right exact sequence
 $P_0\to Q_0\to A$.  It follows by diagram chasing that the map
 $\pi_0\:[X,X']\to\Hom(A,A')$ is surjective, and that the
 kernel $K(X,X')$ is the cokernel of the map
 $f_0^*\:\Hom(Q_0,C')\to\Hom(P_0,C')$.  As the sequence
 $E=(P_0\to Q_0\to A)$ is a free resolution of $A$, we see that this
 cokernel is just $\Ext(A,C')$.  More precisely, we can define
 $\lm\:\Hom(P_0,C')\to\Ext(A,C')$ by $\lm(u)=u_*([E])$, and standard
 homological algebra says that the sequence
 \[ \Hom(Q_0,C') \xra{f_0^*} \Hom(P_0,C') \xra{\lm} \Ext(A,C') \]
 is right exact.  Consider the diagram 
 \[ \xymatrix{
     [\Sg Q,X'] \ar[r]^{f^*} \ar[d]_{\pi_1}^\simeq &
     [\Sg P,X'] \ear[r]^{h^*} \ar[d]_{\pi_1}^\simeq &
     K(X,X') \ar[d]^\om \\
     \Hom(Q_0,C') \ar[r]_{f_0^*} &
     \Hom(P_0,C') \ear[r]_\lm &
     \Ext(A,C')
    }
 \]
 To see that the right hand square commutes, note that
 $h\in K(X,\Sg P)$ and $\om(h)=[E]$.  It follows that for
 $u\:\Sg P\to X'$ we have
 \[ \om(h^*(u)) = \om(uh) = \om(u_*h) = \pi_0(u)_*\om(h) =
      \pi_0(u)_* [E] = \lm(\pi_0(u))
 \]
 as required.  As both rows are right exact, we see that
 $\om\:K(X,Y)\to\Ext(A,C')$ is an isomorphism as claimed.
\end{proof}

\begin{corollary}\label{cor-moore-maps}
 Let $X$ be a Moore spectrum, and let $X'$ be an arbitrary spectrum.
 Then there is a natural short exact sequence
 \[ 2.\Ext(\pi_0(X),\pi_1(X')) \xra{} [X,X'] \xra{} \EEED(G(X),G(X')). \]
\end{corollary}
\begin{proof}
 This follows from the Proposition together with
 Proposition~\ref{prop-om-xi}. 
\end{proof}

\begin{corollary}\label{cor-CW-maps}
 Let $X$ be a CW spectrum with only $0$-cells and $1$-cells, and let
 $X'$ be an arbitrary spectrum.  Then there is a natural short exact
 sequence 
 \[ 2.\Ext(\pi_0(X),\pi_1(X')) \xra{} [X,X'] \xra{G} \EEED(G(X),G(X')).
 \]
\end{corollary}
\begin{proof}
 The assumption on $X$ is that there exists a cofibration sequence
 $P\to Q\to X\to\Sg P$ where $P$ and $Q$ are again free Moore spectra,
 but the map $\pi_0(P)\to\pi_0(Q)$ need not be injective.  Let $R_0$
 and $F_0$ be the kernel and image of this map.  These are subgroups
 of the free abelian groups $\pi_0(P)$ and $\pi_0(Q)$, so they are
 again free.  We have a short exact sequence
 $R_0\xra{}\pi_0(P)\xra{}F_0$, which must split as $F_0$ is free.  We
 know that $\pi_0\:\FreeMoore\to\FreeAb$ is an equivalence, so we have
 a parallel splitting $P=R\Wedge F$ with $\pi_0(R)=R_0$ and
 $\pi_0(F)=F_0$, and we find that the map $P\to Q$ is zero on $R$.
 Now let $Y$ be the cofibre of the map $F\to Q$.  We find that $Y$ is
 a Moore spectrum, and that $X=Y\Wedge\Sg R$.  Thus
 Corollary~\ref{cor-moore-maps} gives the claim for $Y$, and we need
 only check the claim for $\Sg R$.  As $\pi_0(\Sg R)=0$ the first term
 in the short exact sequence vanishes, so the claim is just that the
 map $G\:[\Sg R,X']\to\EEED(G(\Sg R),G(X'))$ is bijective.  Here $R$
 is a free Moore spectrum and so is a wedge of copies of $S$, so we
 can reduce to the case $R=S$.  Here the claim is that
 $\pi_1(X')\simeq[G(S^1),G(X')]$.  This is easy to see directly, or
 one can appeal to the Yoneda Lemma as in Corollary~\ref{cor-sixteen}.
\end{proof}

We can now prove as promised that $F\:\Moore\to\EMD$ is an
equivalence. 

\begin{proof}[Proof of Theorem~\ref{thm-moore-strong}]
 Suppose that $X$ and $X'$ are both Moore spectra.  Then
 $2.\pi_1(X')=0$ by Lemma~\ref{lem-F}, so
 $2.\Ext(\pi_0(X),\pi_1(X'))=0$.  Thus, Corollary~\ref{cor-moore-maps}
 tells us that $[X,X']=\EEED(G(X),G(X'))=\EMD(F(X),F(X'))$, so
 $F\:\Moore\to\EMD$ is full and faithful.  It is also essentially
 surjective by Lemma~\ref{lem-F-reflects}, so it is an equivalence.
\end{proof}

\begin{corollary}\label{cor-G-surjective}
 The functor $G\:\Spectra\to\EEED$ is essentially surjective.
\end{corollary}

Here we give a proof using Moore spectra.
Proposition~\ref{prop-postnikov} will give an alternative proof using
two-stage Postnikov systems.

\begin{proof}
 Consider an object $N=(A,B,C,\eta,\chi,\psi)\in\EEED$.  Choose Moore
 spectra $W$, $X$ and $Y$ with $\pi_0(W)=A/2$ and $\pi_0(X)=C$ and
 $\pi_0(Y)=A$.  Recall that multiplication by $\eta$ gives an
 isomorphism $A/2\to\pi_1(Y)$.  Using
 Proposition~\ref{prop-moore-maps} and the isomorphism
 $[\Sg W,T]=[W,\Sg^{-1}T]$ we see that there is a map
 $f\:\Sg W\to\Sg X\Wedge Y$ such that the induced map
 $\pi_1(f)\:A/2\to C\oplus A/2$ is $a\mapsto(\eta(a),a)$.  Let $Z$ be
 the cofibre of $f$.  A straightforward calculation with the long
 exact sequence from the defining cofibration sequence gives
 $\pi_0(Z)=A$ and $\pi_1(Z)=C$, with multiplication by $\eta$ being
 the originally given map $\eta\:A\to C$.  This gives an isomorphism
 $u\:\pi N\to\pi G(Z)$ in $\ED$.  We know from
 Proposition~\ref{prop-pi} that $\pi$ is full, so we can choose
 $v\:N\to G(Z)$ with $\pi(v)=u$.  The same proposition tells us that
 $\pi$ reflects isomorphism, and $u$ is an isomorphism, so $v$ is an
 isomorphism.  Thus, $N$ is in the essential image of $G$, as
 required.
\end{proof}

We conclude this section by justifying Remark~\ref{rem-zt-moore}.

\begin{lemma}\label{lem-zt}
 There is an element $\zt\in\pi_2(S/2)$ with $\bt\zt=\eta\:S^2\to S^1$
 and $2\zt=\rho\eta^2$ and $\pi_2(S/2)=(\Z/4)\zt$.
\end{lemma}
This is classical but we include a proof for completeness.
\begin{proof}
 From the cofibration sequence defining $S/2$ we obtain an exact
 sequence  
 \[ \pi_2(S) \xra{2} \pi_2(S) \xra{\rho_*} \pi_2(S/2) \xra{\bt_*} 
     \pi_1(S) \xra{2} \pi_1(S)
 \]
 or equivalently 
 \[ \xymatrix{
     (\Z/2)\eta^2 \ar[r]^{2=0} &
     (\Z/2)\eta^2 \mar[r]^{\rho_*} & 
     \pi_2(S/2) \ear[r]^{\bt_*} &
     (\Z/2)\eta \ar[r]^{2=0} &
     (\Z/2)\eta.
    }
 \]
 It follows that there exists $\zt$ with $\bt\zt=\eta$, and that
 $\pi_2(S/2)=\{0,\zt,\rho\eta^2,\zt+\rho\eta^2\}$.  We also know that
 $2.1_{S/2}=\rho\eta\bt$, so $2\zt=\rho\eta\bt\zt=\rho\eta^2$, so in
 fact $\pi_2(S/2)=\{0,\zt,2\zt,3\zt\}$ and $\zt$ has order $4$.
\end{proof}
\begin{remark}
 There is some indeterminacy in the construction of $\zt$, but at the
 end of the argument we see that the only ambiguity is that $\zt$
 could be replaced by $-\zt$.  We do not know of any other approach
 that eliminates this ambiguity.
\end{remark}

\begin{lemma}\label{lem-zt-moore}
 For any Moore spectrum $X$ the map $\zt^*\:[S/2,X]\to\pi_2(X)$ is an
 isomorphism. 
\end{lemma}
\begin{proof}
 First, if $X=S^1$ we have $[S/2,X]=(\Z/2)\bt$ and
 $\pi_2(X)=(\Z/2)\eta=(\Z/2)\bt\zt$, so $\zt^*\:[S/2,X]\to\pi_2(X)$ is
 an isomorphism.  Similarly, if $X=S$ then $\pi_2(X)=(\Z/2)\eta^2$ and
 $[S/2,X]=(\Z/2)\eta\bt=(\Z/2)$ and
 $\pi_2(X)=(\Z/2)\eta^2=(\Z/2)\eta\bt\zt$, so again
 $\zt^*\:[S/2,X]\to\pi_2(X)$ is an isomorphism.  Now choose a
 presentation $P\xra{f}Q\xra{g}X\xra{h}\Sg P$.  As $P$ is a wedge of
 copies of $S$ and $\Sg P$ is a wedge of copies of $S^1$ we find that 
 the maps $\zt^*\:[S/2,P]\to\pi_2(P)$ and
 $\zt^*\:[S/2,\Sg P]\to\pi_2(\Sg P)$ are isomorphisms, and similarly
 for $Q$.  We can now apply the Five Lemma to the diagram 
 \[ \xymatrix{
     [S/2,P] \ar[r]^{f_*} \ar[d]_{\zt^*}^\simeq &
     [S/2,Q] \ar[r]^{g_*} \ar[d]_{\zt^*}^\simeq &
     [S/2,X] \ar[r]^{h_*} \ar[d] &
     [S/2,\Sg P] \ar[r]^{f_*} \ar[d]_\simeq^{\zt^*} &
     [S/2,\Sg Q] \ar[d]_\simeq^{\zt^*}  \\
     \pi_2(P) \ar[r]_{f_*} & 
     \pi_2(Q) \ar[r]_{g_*} & 
     \pi_2(X) \ar[r]_{h_*} & 
     \pi_1(P) \ar[r]_{f_*} &
     \pi_1(Q) 
    }
 \]
 to see that $\zt^*\:[S/2,X]\to\pi_2(X)$ is also an isomorphism.
\end{proof}

\section{Two-stage Postnikov systems}
\label{sec-postnikov}

\begin{definition}
 We write $\Postnikov$ for the category of spectra $X$ such that
 $\pi_k(X)=0$ for $k\not\in\{0,1\}$.
\end{definition}

\begin{proposition}\label{prop-postnikov}
 The functor $G\:\Postnikov\to\EEED$ is full and essentially
 surjective, and it reflects isomorphisms.  If $X,X'\in\Postnikov$
 then there is a natural short exact sequence
 \[ 2.\Ext(\pi_0(X),\pi_1(X')) \xra{} [X,X'] \xra{G} \EEED(G(X),G(X')).
 \]
\end{proposition}

The proof will follow after some preliminary results.

\begin{lemma}\label{lem-HA}
 If $X$ is $(-1)$-connected then the maps
 \[ [X,HA'] \xra{G} \EEED(G(X),G(HA')) \xra{\pi}
     \ED(\pi G(X),\pi G(HA')) \xra{\al} \Hom(\pi_0(X),A')
 \]
 are all isomorphisms.
\end{lemma}
\begin{proof}
 As $\gm G(HA')=\pi_1(HA')=0$, Proposition~\ref{prop-pi} tells us that
 \[ \EEED(G(X),G(HA'))=\ED(\pi G(X),\pi G(HA')), \]
 and using $\gm G(HA')=0$ again we see that this is the same as
 $\Hom(\pi_0(X),A')$.  The Universal Coefficient Theorem gives a short
 exact sequence
 \[ \Ext(H_{-1}(X),A') \to [X,HA']=H^0(X;A') \to \Hom(H_0(X),A'). \]
 As $X$ is $(-1)$-connected we have $H_{-1}(X)=0$ and
 $H_0(X)=\pi_0(X)$ so this collapses to an isomorphism
 $[X,HA']\to\Hom(\pi_0(X),A')=\EEED(G(X),G(HA'))$ as required.  
\end{proof}

\begin{lemma}\label{lem-SHC}
 If $X$ is $(-1)$-connected then there is a short exact sequence
 \[ 2.\Ext(\pi_0(X),C') \xra{} [X,\Sg HC'] \xra{G} \EEED(G(X),G(\Sg HC')).
 \]
\end{lemma}
\begin{proof}
 In view of Proposition~\ref{prop-om-xi}, it will suffice to give a
 short exact sequence 
 \[ \Ext(\pi_0(X),C') \xra{} [X,\Sg HC'] \xra{\pi G}
     \ED(\pi G(X),\pi G(\Sg HC')).
 \]
 The Universal Coefficient Theorem gives a short exact sequence 
 \[ \Ext(H_0(X),C') \xra{} [X,\Sg HC'] \xra{} \Hom(H_1(X),C'). \]
 We also know from Lemma~\ref{lem-H-one} that $H_1(X)$ is the cokernel
 of the map $\eta\:\pi_0(X)\to\pi_1(X)$, so $\Hom(H_1(X),C')$ is the
 same as the group of maps $h$ making the diagram 
 \[ \xymatrix{
      \pi_0(X) \ar[r]^\eta \ar[d] & \pi_1(X) \ar[d]^h \\
      0 \ar[r] & C' 
    }
 \]
 commute, which is $\ED(\pi G(X),\pi G(\Sg HC'))$.
\end{proof}

\begin{lemma}\label{lem-HH}
 For any abelian groups $A$ and $C$ we have $H_1HA=0$ and $H_2HA=A/2$
 and $[HA,\Sg^2HC]=H^2(HA;C)=\Hom(A/2,C)$.  Moreover, if $X$ is the
 fibre of a map $h\:HA\to\Sg^2HC$ corresponding to a homomorphism
 $\sg\:A/2\to C$, then the map 
 \[ \eta \: A = \pi_0(X) \to \pi_1(X) = C \]
 is just the composite of $\sg$ with the projection $A\to A/2$. 
\end{lemma}
\begin{proof}
 First, Lemma~\ref{lem-H-one} tells us that $H_1HA$ is the cokernel of
 $\eta\:\pi_0HA\to\pi_1HA$ but $\pi_1HA=0$ so $H_1HA=0$.  For the
 second claim, we reuse the spectrum $Q$ from Lemma~\ref{lem-Q}.  As
 the fibre of the map $Q\to H$ is $2$-connected we see that
 $H_2HA=\pi_2(Q\Smash HA)=H_2(Q;A)$.  The cofibration sequence
 $S\to Q\to S^2/2\to S^1$ gives $H_2(Q;A)=H_2(S^2/2;A)=H_0(S/2;A)$, and
 the Universal Coefficient Theorem gives a short exact sequence
 \[ H_0(S/2)\ot A \xra{} H_0(S/2;A) \xra{} \Tor(H_{-1}(S/2),A), \]
 which collapses to an isomorphism $H_0(S/2;A)=A/2$.  Finally, the
 cohomological UCT gives a short exact sequence
 \[ \Ext(H_1HA,C) \to H^2(HA;C) \to \Hom(H_2HA,C) \]
 which collapses to an isomorphism
 $H^2(HA;C)=\Hom(H_2HA,C)=\Hom(A/2,C)$ as claimed.

 We now see in particular that $[H,\Sg^2 H/2]=\Z/2$, so there is a
 unique nontrivial map $h_1\:H\to\Sg^2H/2$.  We can also compose the
 projection $H\to H/2$ with the Steenrod operation
 $\Sq^2\:H/2\to\Sg^2H/2$ to get a map $h_2\:H\to\Sg^2H/2$.  Standard
 calculations show that the reduced cohomology of $\C P^2$ is
 $\Z\{x,x^2\}$ with $|x|=2$, and $\Sq^2$ sends the mod $2$ reduction
 of $x$ to the mod $2$ reduction of $x^2$.  This shows that
 $h_2\neq 0$ and so $h_2=h_1$.  Moreover, $\C P^2$ is just the cofibre
 of $\eta\:S^3\to S^2$, so $S/\eta=\Sg^{-2}\C P^2$.  The above
 calculation therefore gives a commutative diagram as follows:
 \[ \xymatrix{
     S/\eta \ar[d]_{a_1} \ar[r]^\dl & S^2 \ar[d]^{c_1} \\
     H \ar[r]_{h_1} & \Sg^2 H/2.
    }
 \]
 
 Now suppose we have a general map $h\:HA\to\Sg^2HC$, corresponding to
 $\sg\:A/2\to C$, and we form a cofibration sequence
 \[ \Sg HC \xra{f} X \xra{g} HA \xra{h} \Sg^2HC. \]
 Consider an element $a\in A$ and the corresponding
 element $c=\sg(a+2A)\in C$ (which automatically has $2c=0$).  We have
 a map $\mu_a\:\Z\to A$ given by $\mu_a(n)=na$, and a map
 $\nu_c\:\Z/2\to C$ given by $\nu_c(1)=c$.  Consider the following
 diagram: 
 \[ \xymatrix{
     S^1 \ar[r]^\eta \ar[dd]_c & 
     S \ar[r] \ar@{.>}[dd]^{a'} &
     S/\eta \ar[d]_{a_1} \ar[r]^\dl &
     S^2 \ar[d]^{b_1} \\ & & 
     H \ar[r]_{h_1} \ar[d]_{H(\mu_a)} &
     \Sg^2H/2 \ar[d]^{H(\nu_c)} \\
     \Sg HC \ar[r]_f &
     X \ar[r]_g &
     HA \ar[r]_h &
     \Sg^2HC.     
    }
 \]
 We have just shown that the top right square commutes.  As our
 identification of $[HA,\Sg^2HC]$ is natural, the bottom right square
 commutes as well.  As the top and bottom rows are cofibration
 sequences, we can choose $a'$ making everything commute.  If we use
 $f_*$ to identify $\pi_0(X)$ with $A$, and $g_*$ to identify
 $\pi_1(X)$ with $C$, then the conclusion is that $\eta a=c$.  This
 shows that $\eta\:\pi_0(X)\to\pi_1(X)$ is essentially the same as
 $\sg$, as claimed.
\end{proof}

\begin{lemma}\label{lem-postnikov-maps}
 Suppose that $X$ is $(-1)$-connected and $X'\in\Postnikov$.  Then
 there is a left exact sequence
 \[ \Ext(\pi_0(X),\pi_1(X')) \xra{} [X,X'] \xra{\pi G} \ED(G(X),G(X')).
 \] 
\end{lemma}
\begin{proof}
 We will use the notation
 \begin{align*}
  M &= \pi G(X) = (A \xra{\eta} C) \\
  M' &= \pi G(X') = (A' \xra{\eta} C').
 \end{align*}
 Put $A'=\pi_0(X')$ and $C'=\pi_1(X')$, so we have a Postnikov
 cofibration  sequence
 \[ \Sg HC'\xra{f}X'\xra{g}HA'\xra{h}\Sg^2 HC'. \]
 (This is usually constructed as a fibration sequence, but in the
 stable category a fibration sequence can be converted to a
 cofibration sequence by changing the sign of any one of the maps.) 
 This gives a commutative diagram 
 \[ \xymatrix{
  \Ext(A,C') \ar[r] \mar[d] &
  K(X,X') \mar[d] \ar[r]^\om &
  \Ext(A,C') \\
  [X,\Sg HC'] \ear[d]_{\pi G} \ar[r]^{f_*} &
  [X,X'] \ar[d]_{\pi G} \ar[r]^{g_*} &
  [X,HA] \ar[d]^{\pi G}_\simeq \\
  \ED(M,(0\to C')) \mar[r]_{f_*} &
  \ED(M,M') \ar[r]_{g_*} &
  \ED(M,(A'\to 0)). 
 } \]
 The middle row is exact as it arises from a cofibration sequence.
 The bottom row is left exact by inspection of the definitions.  The
 first column is short exact by Lemma~\ref{lem-SHC}.  The middle
 column is left exact by definition of $K(X,X')$.  The map
 $[X,HA]\to\ED(M,A'\to 0)$ is an isomorphism by Lemma~\ref{lem-HA}.
 We leave it to the reader to check that the composite of the top row
 is the identity.  It follows that the map $\Ext(A,C')\to K(X,X')$ is
 injective, and by chasing the diagram we see that it is also
 surjective.
\end{proof}

\begin{proof}[Proof of Proposition~\ref{prop-postnikov}]
 Suppose that $X,X'\in\Postnikov$, and use the usual notation
 $A=\pi_0(X)$ and so on.  Suppose we are given a commutative square
 \[ \xymatrix{
     A \ar[r]^\eta \ar[d]_f & C \ar[d]^h \\
     A' \ar[r]_{\eta'} & C',
    }
 \]
 corresponding to a morphism $(f,h)\in\ED(\pi G(X),\pi G(X'))$.
 Consider the following diagram:
 \[ \xymatrix{
      \Sg HC \ar[r] \ar[d]_{\Sg H h} &
      X \ar@{.>}[d]^p \ar[r] & 
      HA \ar[d]_{H f} \ar[r]^d &
      \Sg^2HC \ar[d]^{\Sg^2Hh} \\
      \Sg HC' \ar[r] &
      X' \ar[r] & 
      HA' \ar[r]_{d'} &
      \Sg^2HC'
    }
 \]
 The rows are the standard Postnikov cofibration sequences for $X$ and $X'$.  As
 $d$ and $d'$ correspond to $\eta$ and $\eta'$ as in
 Lemma~\ref{lem-HH}, we see that the right hand square commutes, so we
 can choose $p$ making everything commute; we then have
 $\pi G(p)=(f,h)$.   This shows that $\pi G\:\Postnikov\to\ED$ is
 full.  If $\pi G(p)$ is an isomorphism then $\pi_k(p)$ is an
 isomorphism for $k\in\{0,1\}$ (by assumption) and also for
 $k\not\in\{0,1\}$ (as in that case $\pi_k(X)=0=\pi_k(X')$); so $p$ is
 an equivalence.  This shows that $\pi G$ reflects isomorphism, and
 \emph{a fortiori} $G$ has the same property.  Next, given an object
 $(A\xra{\eta}C)\in\ED$ we note that $\eta$ factors through $A/2$ and
 so gives a map $h\:HA\to\Sg^2HC$.  The fibre $Fh$ is then an object
 of $\Postnikov$ with $\pi G(Fh)\simeq(A\to C)$; so
 $\pi G\:\Postnikov\to\EEED$ is essentially surjective.

 We also see from Lemma~\ref{lem-postnikov-maps} that the map
 $\om\:K(X,X')\to\Ext(A,C')$ is an isomorphism.  We deduce using
 Proposition~\ref{prop-om-xi} that there is a natural short exact
 sequence 
 \[ 2.\Ext(A,C') \to [X,X'] \to \EEED(G(X),G(X')). \]
 In particular, the functor $G\:\Postnikov\to\EEED$ is full.  

 Finally, suppose we have an object $N\in\EEED$.  We have seen that
 $\pi G$ is essentially surjective, so we can choose $X\in\Postnikov$
 and an isomorphism $p_0\:\pi G(X)\to\pi N$.  As $\pi$ is full we can
 choose $p\:G(X)\to N$ lifting $p_0$, and as $\pi$ reflects
 isomorphisms we see that $p$ is an isomorphism.  This proves that $G$
 is essentially surjective.
\end{proof}

\section{Duality}
\label{sec-duality}

We next discuss two kinds of duality for the category $\EEED$.

\begin{lemma}\label{lem-Dl}
 Let $\CJ$ be as in Definition~\ref{defn-CJ}.  Then there is a functor
 $\Dl\:\CJ^\opp\to\CJ$ given by 
 \begin{align*}
  \Dl(a) &= c & \Dl(b) &= b & \Dl(c) &= S \\
  \Dl(\rho) &= \bt & \Dl(\eta) &= \eta & \Dl(\bt) &= \rho.
 \end{align*}
 Moreover, this satisfies $\Dl^2=1$.  
\end{lemma}
\begin{proof}
 A straightforward check of definitions.
\end{proof}

This is connected with topology as follows:

\begin{lemma}\label{lem-Dl-T}
 If we define $\Xi\:(\CJ\tm\CJ)^\opp\to\Ab$ by
 $\Xi(x,y)=[T(x)\Smash T(y),S^1]$, then there are natural isomorphisms 
 \[ \CJ(x,\Dl(y))\simeq \Xi(x,y) \simeq \CJ(y,\Dl(x)). \]
\end{lemma}
\begin{proof}
 It is equivalent to claim that there are natural isomorphisms
 $T(\Dl(x))\simeq F(T(x),S^1)$.  This is clear for $x\in\{a,c\}$ (so
 $T(x)\in\{S,S^1\}$).  Next, we can apply $F(-,S^1)$ to the
 cofibration sequence
 \[ S\xra{2}S\xra{\rho}S/2\xra{\bt}S^1\xra{2}S^1 \]
 to get a fibration sequence
 \[ S\xra{2}S\xra{\bt^*}F(S/2,S^1)\xra{\rho^*}S^1\xra{2}S^1. \]
 In principle we need to change a sign to convert this fibration to a
 cofibration, but $\rho=-\rho$ and $\bt=-\bt$ so we need not worry
 about this.  It follows from the essential uniqueness of cofibration
 sequences that there is an equivalence $\nu\:S/2\to F(S/2,S^1)$
 making the diagram on the left commute:
 \[ \xymatrix{
     S \ar[r]^\rho \ar[d]_{\bt^*} &
     S/2 \ar[dl]_\nu \ar[d]^\bt & &
     (S/2)^{(2)} \ar[dr]^{\nu^\#} &
     S/2 \ar[l]_{\rho\Smash 1} \ar[d]^\bt \\
     F(S/2,S^1) \ar[r]_{\rho^*} &
     S^1 & &
     S/2 \ar[u]^{\rho\Smash 1} \ar[r]_\bt &
     S^1
    }
 \]
 (One can check that $\nu$ has order $4$, and that the only
 indeterminacy is that we could replace $\nu$ by $-\nu$.)  The map
 $\nu$ is adjoint to a map $\nu^\#\:S/2\Smash S/2\to S^1$, and the
 adjoint form of the left hand diagram is shown on the right.  It
 follows that everything fits together as described.
\end{proof}

We can use this to define a Brown-Comenetz type duality as follows.  
\begin{definition}\label{defn-BC-dual}
 For $N\in\EED=[\CJ^\opp,\Ab]$ we define $JN\:\CJ^\opp\to\Ab$ by
 $JN(x)=\Hom(N(\Dl x),\Q/\Z)$.  More explicitly, we have 
 \[ J(B\xra{\psi}A\xra{\eta}C\xra{\chi}B) = 
     (B^*\xra{\chi^*}C^*\xra{\eta^*}A^*\xra{\psi^*}B^*)
 \]
 (where $U^*$ denotes $\Hom(U,\Q/\Z)$).
\end{definition}

\begin{remark}\label{rem-BC-dual}
 As the group $\Q/\Z$ is divisible, the functor $U\mapsto U^*$ is
 exact, so the functor $J\:\EED^\opp\to\EED$ preserves $\EEED$.  It is
 also clear that there is a natural map $N\to J^2(N)$ which is an
 isomorphism when the groups $A$, $B$ and $C$ are all finite.
\end{remark}

\begin{proposition}\label{prop-BC-dual}
 Let $IX$ denote the Brown-Comenetz dual of a spectrum $X$, which is
 characterised by a natural isomorphism
 $[W,IX]\simeq\Hom(\pi_0(W\Smash X),\Q/\Z)$ for all spectra $W$.  Then
 $G(\Sg IX)=J(G(X))$.
\end{proposition}
\begin{proof}
 For $u\in\CJ$ we have 
 \[ G(\Sg IX)(u) = [\Sg^{-1}T(u),IX] =
     \Hom(\pi_0(\Sg^{-1}T(u)\Smash X),\Q/\Z).
 \]
 We also have $T(u)=T(\Dl^2 u)=F(T(\Dl u),S^1)$, so
 $\Sg^{-1}T(u)\Smash X=F(T(\Dl u),X)$, so
 \[ \pi_0(\Sg^{-1}T(u)\Smash X)=[T(\Dl u),X]=G(X)(\Dl u). \]
 Putting this together gives
 $G(\Sg IX)(u)=\Hom(G(X)(\Dl u),\Q/\Z)=J(G(X))(u)$.
\end{proof}

We now discuss a different construction that is more analogous to
Spanier-Whitehead duality.  It does not work very well, but it is
interesting that it works at all.

\begin{definition}\label{defn-Dl-EED}
 Recall that $F_x$ denotes the representable functor
 $\CJ(-,x)\:\CJ^\opp\to\Ab$.  We define $\Dl\:\EED^\opp\to\EED$ by 
 \[ \Dl(N)(x) = \EED(N,F_{\Dl(x)}). \]
\end{definition}

Note that a map $M\to\Dl N$ consists of a natural system of maps
$M(x)\to\EED(N,F_{\Dl(x)})$, or equivalently a natural system of maps
$M(x)\to\Hom(N(y),\CJ(y,\Dl(x)))$, or equivalently a natural system of
maps 
\[ M(x) \otimes N(y) \to \Xi(x,y). \] 
Using this we see that $\EED(M,\Dl(N))\simeq\EED(N,\Dl(M))$, so $\Dl$
is its own adjoint.  In particular, we have natural maps
\[
 M(x) \ot (\Dl M)(y) =
  M(x)\ot\EEED(M,F_{\Dl y}) \xra{1\ot\text{eval}_x}
  M(x)\ot\Hom(M(x),\Xi(y,x)) \xra{\text{eval}} \Xi(y,x)\simeq \Xi(x,y)
\]
and these correspond to a natural map $\kp\:M\to\Dl^2M$.  If we regard
this as a morphism in $\EED$, this is the unit of our adjunction; if we
regard it instead as a morphism $\Dl^2M\to M$ in $\EED^\opp$, it is the
counit. 

We can also use the Yoneda lemma to see that 
\[ \Dl(F_x)(y) = \EED(F_x,F_{\Dl(y)}) = 
    F_{\Dl(y)}(x) = \CJ(x,\Dl(y)) \simeq \CJ(y,\Dl(x)) = F_{\Dl(x)}(y) 
\]
so $\Dl(F_x)=F_{\Dl(x)}$.

\begin{proposition}\label{prop-kappa}
 There is a natural map $G(F(X,S^1))\to\Dl(G(X))$, which is an
 isomorphism when $X$ is any wedge of copies of $S$, $S^1$ and $S/2$. 
\end{proposition}
\begin{proof}
 Using the evaluation map $\ev\:F(X,S^1)\Smash X\to S^1$ we get maps
 \begin{align*}
  G(F(X,S^1))(x)\ot G(X)(y) &= [T(x),F(X,S^1)]\ot [T(y),X] \\
   &\to [T(x)\Smash T(y),F(X,S^1)\Smash X] \\
   &\xra{\ev_*} [T(x)\Smash T(y),S^1] = \Xi(x,y).
 \end{align*}
 As explained above this gives a map
 $\kp_X\:G(F(X,S^1))\to\Dl(G(X))$.  In particular, when
 $X=T(u)\in\{S,S^1,S/2\}$ we have $F(X,S^1)=T(\Dl(u))$ so $\kp_X$ is a
 map from $G(T(\Dl u))=F_{\Dl u}$ to $\Dl(F_u)$; we leave it to the
 reader to check that this is the same as the isomorphism established
 previously.  Moreover, the domain and codomain of $\kp$ are both
 functors that convert wedges to products, so we see that $\kp_X$ is
 an isomorphism when $X$ is any wedge of copies of $S$, $S^1$ and
 $S/2$.
\end{proof}

We can make the functor $\Dl$ more explicit as follows:
\begin{proposition}\label{prop-Dl-formulae}
 If $N=(A,B,C,\eta,\chi,\psi)\in\EED$ then
 $\Dl(N)\simeq (A',B',C',\eta',\chi',\psi')$, where $B'=\Hom(B,\Z/4)$,
 and $A'$ is the set of pairs $(g,h)$ making the left hand diagram
 below commute, and $C'$ is the set of pairs $(f,g)$ making the right
 hand diagram commute.
 \[ \xymatrix{
       C \ar[r]^\chi \ar[d]_h & B \ar[d]^g & & 
       A \ar[r]^{\chi\eta} \ar[d]_f & B \ar[d]^g \\
       \Z \ear[r] & \Z/2 & & \Z \ear[r] & \Z/2.
    }
 \]
 There is a unique nonzero homomorphism $t\:\Z/2\to\Z/4$ and a unique
 nonzero homomorphism $s\:\Z/4\to\Z/2$, and using these 
 we can describe $\eta'$, $\chi'$ and $\phi'$ as follows:
 \begin{align*}
  \psi'(g) &= (sg,0) \\
  \eta'(g,h) &= (0,g) \\
  \chi'(f,g) &= tg.
 \end{align*}
\end{proposition}
\begin{proof}\ \\
 By definition we have
 \[ A' = \Dl(N)(a) = \EED(N,F_{\Dl(a)}) = \EED(N,F_c). \] 
 Looking back to Corollary~\ref{cor-sixteen}, we see that
 this is the set of triples $(f,g,h)$ making the following diagram
 commute: 
 \[ \xymatrix{
      B \ar[r]^\psi \ar[d]_g & 
      A \ar[r]^\eta \ar[d]_f & 
      C \ar[r]^\chi \ar[d]^h &
      B \ar[d]^g \\
      \Z/2 \ar[r] & 
      0 \ar[r] &
      \Z \ear[r] &
      \Z/2.
    }
 \]
 Of course $f$ must be zero and the first square commutes
 automatically.  Note also that for any homomorphism $h\:C\to\Z$ we
 have $2h\eta = h\circ(2\eta)=0$ but $\Z$ is torsion free so
 $h\eta=0$.  Thus, the second square also commutes automatically and
 we need only consider the third one.  This gives the stated
 description of $A'$.  

 Similarly, $C'$ is the set of triples $(f,g,h)$ making the diagram 
 \[ \xymatrix{
      B \ar[r]^\psi \ar[d]_g & 
      A \ar[r]^\eta \ar[d]_f & 
      C \ar[r]^\chi \ar[d]^h &
      B \ar[d]^g \\
      \Z/2 \ar[r]_0 & 
      \Z \ear[r] &
      \Z/2 \ar[r]_1 &
      \Z/2.
    }
 \]
 The first square commutes automatically because $2\psi=0$.  The third
 square forces $h$ to be $g\chi$, so we need not mention $h$
 explicitly.  The only remaining condition is that $g\chi\eta$ should
 be the mod two reduction of $f$, which is the stated description of
 $C'$.  

 Next, $B'$ is the set of triples $(f,g,h)$ making the diagram on the
 left (and therefore also the diagram on the right) commute:
 \[ \xymatrix{
      B \ar[r]^\psi \ar[d]_g & 
      A \ar[r]^\eta \ar[d]_f & 
      C \ar[r]^\chi \ar[d]^h &
      B \ar[d]^g & & 
      C/2 \ar[d]^h \ar@{>->}[r] & 
      B \ar[d]^g \ear[r] &
      A[2] \ar[d]^f \\
      \Z/4 \ear[r]_s & 
      \Z/2 \ar[r]_1 &
      \Z/2 \mar[r]_t &
      \Z/4 & & 
      \Z/2 \mar[r]_t &
      \Z/4 \mar[r]_s &
      \Z/2.
    }
 \]
 There is a projection $\pi\:B'\to\Hom(B,\Z/4)$ given by
 $(f,g,h)\mapsto g$, and we claim that this is bijective.  First,
 suppose that $g=0$.  The left half of the right hand diagram shows
 that $h=0$, and then the middle square of the left diagram shows that
 $f=h\eta=0$.  This proves that $\pi$ is injective.  Suppose instead
 that we start with an arbitrary map $g\:B\to\Z/4$.  As $2\chi=0$ we
 have $2g\chi=0$ so the map $g\chi\:C\to\Z/4$ factors through
 $(\Z/4)[2]=t(\Z/2)$.  Thus, there is a unique map $h\:C\to\Z/2$ with
 $g\chi=th$, so the right square of the left diagram commutes.  We
 define $f=h\eta\:A\to\Z/2$, which makes the middle square commute.
 We now have $tf\psi=th\eta\psi=g\chi\eta\psi$, and
 $\chi\eta\psi=2.1_B$ so $tf\psi=2g$.  On the other hand,
 $ts=2\:\Z/4\to\Z/4$, so we also have $tsg=2g$.  As $t$ is injective
 and $tf\psi=2g=tsg$ we conclude that $f\psi=sg$, so the left square
 commutes.  We thus have an element $(f,g,h)\in B'$ with
 $\pi(f,g,h)=g$, showing that $\pi$ is actually an isomorphism.  We
 will therefore identify $B'$ with $\Hom(B,\Z/4)$.

 Now consider the picture 
 \[ \xymatrix{
      B \ar[r]^\psi \ar[d]_g & 
      A \ar[r]^\eta \ar[d]_f & 
      C \ar[r]^\chi \ar[d]^h &
      B \ar[d]^g \\ 
      \Z/4 \ear[r]_s \ear[d]_s & 
      \Z/2 \ar[r]_1 \ar[d]_0 &
      \Z/2 \mar[r]_t \ar[d]^0 &
      \Z/4 \ar[d]^s \\
      \Z/2 \ar[r] &
      0 \ar[r] &
      \Z \ear[r]_s &
      \Z/2
    }
 \]
 The second and third rows are $F_{\Dl(b)}=F_b$ and $F_{\Dl(a)}=F_c$,
 and the map between them is $F_{\Dl(\rho)}=F_\bt$.  The map from the
 first to the second row defines a general element of
 $B'=\EED(N,F_b)$.  The map $\psi'\:B'\to A'$ is defined by
 composing vertically to give the system of maps $0\:A\to 0$,
 $sg\:B\to\Z/2$ and $0\:C\to\Z$.  In terms of our slightly more
 compact notation for elements of $B'$ and $A'$, the formula is
 $\psi'(g)=(sg,0)$ as claimed.  We leave it to the reader to check the
 formulae $\eta'(g,h)=(0,g)$ and $\chi'(f,g)=tg$ in the same way.
\end{proof}

\begin{remark}\label{rem-Dl-zero}
 It follows easily from the above that the torsion subgroups in $A'$
 and $C'$ have exponent $2$ (or are trivial).  Thus, if $A$ or $C$ has
 any elements of order greater than $2$, we see that
 $N\not\simeq\Dl^2(N)$.  In fact, if $A$ and $C$ are finite groups of
 odd order (which forces the structure maps to be zero, and $B$ to be
 the zero group) then $\Dl(N)=0$.
\end{remark}

\begin{remark}
 Consider an object $N=(A,B,C,\eta,\chi,\psi)\in\EEED$.  Suppose that
 $A$ and $C$ are elementary abelian groups of finite rank, and that
 $\eta\:A\to C$ is zero.  We then see from
 Proposition~\ref{prop-H-equivalence} that $B=A\oplus C$, and
 Proposition~\ref{prop-Dl-formulae} gives $A'=\Hom(A,\Z/2)$ and
 $B'=C'=\Hom(A,\Z/2)\oplus\Hom(C,\Z/2)$.  Just by comparing orders we
 deduce that the sequence $C=C/2\to B\to A[2]=A$ cannot be short
 exact, so $\Dl(N)\not\in\EEED$.  Thus, the functor
 $\Dl\:\EED^\opp\to\EED$ does not preserve the category $\EEED$.
\end{remark}


\begin{bibdiv}
\begin{biblist}

\bib{ba:ah}{book}{
  author={Baues, Hans~J.},
  title={Algebraic homotopy},
  series={Cambridge Studies in Advanced Mathematics},
  publisher={Cambridge University Press},
  date={1989},
  volume={15},
}

\bib{ba:ht}{book}{
   author={Baues, Hans-Joachim},
   title={Homotopy type and homology},
   series={Oxford Mathematical Monographs},
   note={Oxford Science Publications},
   publisher={The Clarendon Press Oxford University Press},
   place={New York},
   date={1996},
   pages={xii+489},
   isbn={0-19-851482-4},
   review={\MR{1404516 (97f:55001)}},
}

\bib{fr:sh}{incollection}{
  author={Freyd, Peter},
  title={Stable homotopy},
  date={1966},
  booktitle={Proc. conf. categorical algebra (La Jolla, Calif., 1965)},
  publisher={Springer},
  address={New York},
  pages={121\ndash 172},
  review={\MR {35 \#2280}},
}

\bib{gapo:cdc}{article}{
  author={Popesco, Nicolae},
  author={Gabriel, Pierre},
  title={Caract\'erisation des cat\'egories ab\'eliennes avec g\'en\'erateurs et limites inductives exactes},
  language={French},
  journal={C. R. Acad. Sci. Paris},
  volume={258},
  date={1964},
  pages={4188--4190},
  review={\MR {0166241 (29 \#3518)}},
}

\bib{to:oic}{article}{
  author={Toda, Hirosi},
  title={Order of the identity class of a suspension space},
  journal={Ann. of Math. (2)},
  volume={78},
  date={1963},
  pages={300--325},
  issn={0003-486X},
  review={\MR {0156347 (27 \#6271)}},
}

\end{biblist}
\end{bibdiv}
\end{document}